\documentclass {article}
\usepackage{authblk}
\usepackage{etex}
\usepackage[utf8]{inputenc}
\usepackage[english]{babel}
\usepackage{amsmath}
\usepackage{amsthm}
\usepackage{amssymb}
\usepackage{sectsty}
\usepackage{titlesec}
\usepackage{color}
\usepackage[color,matrix,arrow]{xy}
\usepackage{amsgen}
\usepackage{amstext}
\usepackage{amsbsy}
\usepackage{amsopn}
\usepackage{amsfonts}
\usepackage{bm}
\usepackage{eepic}
\usepackage{graphicx}
\usepackage{epsf}
\usepackage{pstricks}
\usepackage{xfrac}
\usepackage{float}
\usepackage{todonotes}
\usepackage{tikz}
\usetikzlibrary{automata, positioning}
\usepackage{enumerate}
\usepackage{eqnarray}
\usepackage{array}
\usepackage{faktor}
\usepackage{ mathrsfs }
\usepackage{enumitem}
\usepackage{caption}
\xyoption{all}
\usetikzlibrary{calc,decorations.pathreplacing,arrows.meta,positioning}

\usepackage{multicol}

\usepackage{pgf}
\usepackage{tikz-cd}
\usetikzlibrary{automata, arrows.meta, positioning,decorations.pathmorphing}

\newcommand{\footrecall}[1]{
} 
\usetikzlibrary{snakes,shapes,arrows,automata}

\usepackage{tikz}
\usetikzlibrary{arrows.meta, decorations.pathmorphing}

\titleformat*{\section}{\large\bfseries}
\titleformat*{\subsection}{\normalsize \bfseries}

\newcommand{\MT}{\mathrm{MT}}

\newcommand{\N}{\mathbb{N}}

\newcommand{\DGeo}{\mathcal{D}\text{-}\Geo}

\newcommand{\Geo}{\text{Geo}}
\newcommand{\FIM}{\text{FIM}}
\newcommand{\ConjGeo}{\text{ConjGeo}}

\newcommand{\UConj}{\mathrm{UConj}}

\theoremstyle{definition}
\newtheorem{theorem}{Theorem}[section]

\newtheorem{definition}[theorem]{Definition}

\newtheorem{proposition}[theorem]{Proposition}

\newtheorem{lemma}[theorem]{Lemma}
\newtheorem{example}[theorem]{Example}
\newtheorem{remark}[theorem]{Remark}

\newcommand{\UConjGeo}{\text{UConjGeo}}

\begin{document}

\title{Conjugacy languages in free inverse monoids}
\author[1]{André Carvalho}
\affil[1]{Centro de Investigação em Matemática e Aplicações (CIMA)
	
	Departamento de Matemática, Escola de Ciências e Tecnologia da Universidade de Évora
	
	Rua Romão Ramalho, 59, 7000–671 Évora, Portugal
	
	\texttt{andre.carvalho@uevora.pt}\\}

\author[2]{Ana-Catarina C. Monteiro}
\affil[2]{Center for Mathematics and Applications (NOVA Math), NOVA School of Science and Technology (NOVA FCT)\\
	\texttt{acatarinacm@gmail.com}}
 \maketitle
 
  \begin{abstract}
We initiate the study of conjugacy languages in free inverse monoids. Motivated by the notion of conjugacy languages in groups and the study of conjugacy in semigroups, we introduce the language of shortest representatives of conjugacy classes and study it for free inverse monoids of rank at least $2$ under the natural notion of conjugacy. We show that, contrary to the free group case, where this language is regular, in a free inverse monoid it is neither context-free nor co-context-free. In the monogenic case, this language is context-free.

We show that this non-context-freeness comes from elements with nontrivial conjugacy classes. We define an equivalence relation as follows: all elements with a nontrivial conjugacy class are related and elements with trivial conjugacy class are only related to themselves. We call this relation $\UConj$. We show that the language consisting of geodesics representing elements whose conjugacy class is trivial is context-free by providing an explicit context-free grammar generating it. 

For groups, we show that the language of minimal representatives of $\UConj$ classes is regular if the group is hyperbolic and for right-angled Artin groups with the standard generating set, it is piecewise testable. For virtually abelian groups, we show that there is a generating set for which this language is piecewise excluding and exhibit an example of a virtually abelian group admitting a generating set for which this language is not regular. 
\end{abstract}

  \section{Introduction}

Conjugacy is one of the fundamental equivalence relations in group theory.
Given a finitely generated group $G$ with generating set $X$, one can study
conjugacy from the point of view of formal language theory by considering
languages of distinguished representatives of its conjugacy classes. In
particular, a word is a \emph{geodesic} if it has minimal length among all
words representing the same element of $G$, while a conjugacy geodesic is a
geodesic whose length is minimal among all representatives of elements in its
conjugacy class. The corresponding conjugacy geodesic languages were introduced
and studied for several classes of groups in~\cite{[CHHR16]}. For example, in
a free group both the language of geodesics and the conjugacy geodesic language
are regular.

The situation outside groups is more subtle, since there are several
nonequivalent notions of conjugacy for semigroups and monoids; see, for
example,~\cite{[KM09],[AKKM17],[Kon18],[AKK19],[ABKKMM26]}. In this paper we
work with $i$-conjugacy. If $S$ is an inverse monoid and $a,b\in S$, we say
that $a$ and $b$ are \emph{$i$-conjugate}, and write $a\sim_i b$, if there
exists $g\in S^1$ such that
\[
    g^{-1}ag=b
    \qquad\text{and}\qquad
    gbg^{-1}=a.
\]
This relation is an equivalence relation and, for inverse semigroups, it
coincides with natural conjugacy~\cite[Theorem~2.6]{[Kon18]}. Throughout the
paper, when we refer simply to conjugacy, we mean this relation, and we denote
the conjugacy class of an element $a$ by $[a]$. The corresponding conjugacy
problem is decidable in free inverse monoids~\cite{[AKK19]}.

Our main object of study is the free inverse monoid $\FIM_X$ on a finite set
$X$. A central tool in the study of free inverse monoids is the representation of elements by
\emph{Munn trees}: finite labelled trees with distinguished initial and final
vertices which completely determine the represented element~\cite{[M72]}.
Moreover, $i$-conjugacy in free inverse monoids admits a geometric
characterization in terms of these distinguished vertices~\cite{[AKK19]}.
This makes Munn trees particularly suitable for studying conjugacy languages.

There is already a significant difference between free groups and free inverse
monoids at the level of ordinary geodesics. In~\cite{[CBM26]}, it is shown
that the language $\Geo(\FIM_X)$ of all geodesic words in a finitely generated
free inverse monoid is context-free and co-context-free, but it is not regular, while in free groups (in fact, in hyperbolic groups \cite{[Can84]}), we have regularity.
It is therefore natural to ask what happens when minimality is imposed not
only within the set of representatives of one element, but within an entire
conjugacy class.

Motivated by the conjugacy geodesic language for groups, we define
\[
    \ConjGeo(\FIM_X)
    =
    \left\{
        w\in\Geo(\FIM_X)
        :
        \ell(w)\leq \ell(u)
        \text{ for every }u\in\tilde X^*
        \text{ such that }[w\pi]=[u\pi]
    \right\},
\]
where $\widetilde X=X\cup X^{-1}$, $\pi\colon\tilde X^*\to\FIM_X$ is the natural epimorphism and
$\ell(w)$ denotes the length of a word $w$. Thus,
$\ConjGeo(\FIM_X)$ consists precisely of the geodesic words having minimal
length within their conjugacy class.

In the monogenic free inverse monoid, geodesic representatives of conjugate
elements have the same length, and hence
\[
    \ConjGeo(\FIM_X)=\Geo(\FIM_X).
\]
Consequently, the conjugacy geodesic language is context-free in this case.
The situation changes sharply as soon as the rank is at least two. Our first
main results show that, for $|X|\geq 2$,
\(
    \ConjGeo(\FIM_X)
\)
is neither context-free nor co-context-free. This contrasts
with the free group case, where the corresponding language is regular
\cite{[CHHR16]}. The proofs use the geometric characterization of conjugacy
in terms of Munn trees together with Ogden's Lemma.

This raises the question of where this additional language-theoretic
complexity comes from. To investigate this, we distinguish elements according
to whether their conjugacy class is trivial. We introduce an equivalence
relation $\UConj$ on $\FIM_X$ in which every element whose conjugacy class is
a singleton forms an equivalence class by itself, while all elements having a
nontrivial conjugacy class belong to one common equivalence class. In other
words, for $a,b\in\FIM_X$,
\[
    a\,\UConj\, b
\]
if either $a=b$ and $[a]=\{a\}$, or both $[a]$ and $[b]$ are nontrivial.
We then consider the corresponding language $\UConjGeo(\FIM_X)$ of geodesic
words having minimal length within their $\UConj$-class.

The distinction between trivial and nontrivial conjugacy classes is
particularly meaningful in free inverse monoids. In a group, an element has a
trivial conjugacy class if and only if it belongs to the center. By contrast,
although the center of a non-monogenic free inverse monoid is trivial, there
are many nonidentity elements whose conjugacy classes are singletons. We give
a simple characterization of these elements in terms of their Munn trees:
apart from the exceptional minimal representatives of length two arising from
the nontrivial $\UConj$-class, a geodesic belongs to
$\UConjGeo(\FIM_X)$ precisely when there is no directed edge with the same
label leaving both the initial and final vertices of its Munn tree.

Using this characterization, we construct an explicit context-free grammar
generating $\UConjGeo(\FIM_X)$.
Thus the obstruction to context-freeness of the full conjugacy geodesic
language comes from elements belonging to nontrivial conjugacy classes.

Finally, we investigate the analogous $\UConj$ language in groups. Since an
element of a group has trivial conjugacy class precisely when it is central,
$\UConjGeo(G)$ differs only by a finite set from the language of geodesic words
representing elements of $Z(G)$. This observation allows us to determine the
language-theoretic complexity of $\UConjGeo$ for several familiar classes of
groups. We prove that $\UConjGeo_X(G)$ is regular for every hyperbolic group
$G$ and every finite generating set $X$. For right-angled Artin groups with
their standard generating sets, we show that $\UConjGeo$ is piecewise
testable. For every finitely generated virtually abelian group, we construct a
finite generating set for which $\UConjGeo$ is piecewise excluding, and hence
regular, complementing the corresponding results for geodesic and conjugacy
geodesic languages in~\cite{[HHR07],[CHHR16]}. We also exhibit a virtually
abelian group with a generating set for which $\ConjGeo$ is regular while
$\UConjGeo$ is not regular. Thus, even for groups, the complexity of
$\UConjGeo$ need not be bounded above by that of $\ConjGeo$ with respect to
the same generating set.

The paper is organized as follows. In Section~\ref{preliminaries}, we recall
the necessary background on formal languages, free inverse monoids, Munn
trees, and conjugacy. In Section~\ref{conj_lang}, we introduce
$\ConjGeo(\FIM_X)$ and prove that it is neither context-free nor
co-context-free when $|X|\geq2$. In Section~\ref{sec: uconj}, we introduce
$\UConj$, characterize the elements with trivial conjugacy class, and prove
that $\UConjGeo(\FIM_X)$ is context-free. In the final section, we study
$\UConjGeo$ for hyperbolic groups, right-angled Artin groups, and virtually
abelian groups.

 \section{Preliminares}\label{preliminaries}
 In this section, we will present preliminary notions and results about formal language theory and free inverse monoids.
 \subsection{Formal languages}
 \subsubsection{Regular languages}
 An alphabet $X$ is a set and its elements are called letters. A word over $X$ is simply a finite sequence of letters of $X$.
Given an alphabet $X$, the Kleene star ($*$) on $X$ generates the set of all words over $X$, i.e.,
$$
X^* = \bigcup_{n \ge 0} X^n,
$$
where $X^0 = \{\varepsilon\}$ and $\varepsilon$ denotes the empty sequence. Together with concatenation, $X^*$ forms the free monoid generated by $X$.

A language $L \subseteq X^*$ is \textit{regular} if it can be constructed  $\emptyset$, $\{\varepsilon\}$, and $\{a\}$ (for all $a \in X$) through a finite sequence of regular operations: union ($\cup$), concatenation ($\cdot$), and Kleene star ($*$). Recall that regular languages are closed under finite unions, intersections, and complementation. Furthermore, the class of regular languages coincides precisely with the languages accepted by deterministic finite automata (DFAs) (for more details, the reader is referred to \cite{[HU79]}).

 	Also, every regular language $L$ satisfies the \textit{Pumping Lemma}, which can be stated as follows.
 	\begin{lemma}[Pumping Lemma for Regular Languages]
 		Let $L$ be a regular language. Then there exists a constant $p \geq 1$ such that every word $w \in L$ with $|w| \geq p$ can be written in the form
 		\[
 		w = xyz
 		\]
 		satisfying the following conditions:
 		\begin{enumerate}
 			\item $|xy| \leq p$;
 			\item $|y| \geq 1$;
 			\item For all $i \geq 0$, the word $x y^i z$ belongs to $L$.
 		\end{enumerate}
 \end{lemma}

 To show that a language is regular, it suffices to write it using the allowed operations from the base elements or to build a DFA that accepts it. The pumping lemma is a most useful tool in proving that certain languages are not regular, as it provides a necessary condition to regularity.
 \subsubsection{Context-free languages}

A \emph{context-free grammar}  is a tuple $(V,A,P,S)$ where 
\begin{itemize}
\item $V$ is a finite set of \emph{variables};
\item $A$ is a set of \emph{terminal symbols} disjoint from $V$;
\item $P$ is a finite subset of  $V\times (V\cup A)^*$. An element of $R$ is a \emph{production};  
\item $S\in V$ is the \emph{starting symbol}.
\end{itemize}

A production $(X,Y)$ is often denoted by $X\to Y$ and usually, when there are two or more productions with the same left-hand side, we denote them in the same line separated by a bar. So, for example, we write $X\to Y\mid Z$ to denote the productions $X\to Y$ and $X\to Z$.

Given $U,W\in (V\cup A)^*$, we write $U\rightarrow W$ if there are $X,Y\in (V\cup A)^*$ and a production $(Z,Z')$ such that $U=XZY$ and $W=XZ'Y$. We call  $U\rightarrow W$ a \emph{derivation}.  Also, we say that $U\rightarrow^* W$ if there is some chain of derivations of the form
$$U=X_0\rightarrow X_1 \rightarrow \cdots \rightarrow X_m=W. $$

The \emph{language of a context-free grammar $G$, $L(G)$}  is simply
$$L(G)=\{w\in A^*\mid S\Rightarrow^* w\}$$ 
and a language is \emph{context-free} if it is the language of some context-free grammar.

For any given variable $K\in V$, we define $L(K)=\{w\in \Sigma^*\, |\, K\to^* w\}$. In particular, $L(G)=L(S).$

The class of context-free languages  is closed under intersection with regular languages and under finite unions, but not under intersection and complement. For further details on context-free languages see, for example, \cite{[Ber79]} or \cite{[HU79]}.

 A language is said to be \emph{co-context-free} if its complement is context-free. As a consequence of the fact that context-free languages are closed under finite unions, it follows that co-context-free languages are closed under finite intersections.
 
 Below we state two necessary conditions for a language to be context-free, which will be particularly useful in this paper for proving that given languages are not context-free.
 
 \begin{lemma}[Pumping Lemma for Context-Free Languages]
 	Let $L$ be a context-free language. Then there exists a constant $p \geq 1$ such that every word $w \in L$ with $|w| \geq p$ can be written in the form
 	\[
 	w = uvxyz
 	\]
 	satisfying the following conditions:
 	\begin{enumerate}
 		\item $|vxy| \leq p$;
 		\item $|vy| \geq 1$;
 		\item For all $i \geq 0$, the word $uv^i x y^i z$ belongs to $L$.
 	\end{enumerate}
 \end{lemma}
 
 \begin{lemma}[Ogden's Lemma]
 	Let $L\subseteq\tilde X^*$ be a context-free language. Then there exists a constant $p \geq 1$ such that for every word $w \in L$ with $|w| \geq p$, and for every choice of at least $p$ distinguished positions in $w$, there exists a decomposition
 	\[
 	w = uvxyz
 	\]
 	satisfying the following conditions:
 	\begin{enumerate}
 		\item The substrings $v$ and $y$ together contain at least one distinguished position;
 		\item The substring $vxy$ contains at most $p$ distinguished positions;
 		\item For all $i \geq 0$, the word $uv^i x y^i z$ belongs to $L$.
 	\end{enumerate}
 \end{lemma}

Similar to the case of regular languages, these lemmas will be used to show that certain languages are not context-free.
 
 An important framework in formal language theory is the \emph{Chomsky hierarchy}, which classifies languages according to their expressive power and the types of grammars that generate them. It consists of four classes: regular, context-free, context-sensitive, and recursively enumerable languages, arranged in a hierarchy by inclusion, where each class properly contains the previous one. In particular, every regular language is context-free, and every context-free language is context-sensitive, and context-sensitive languages are recursively enumerable. We will also consider the class of co-context-free languages, does not coincide with any level of the Chomsky hierarchy.
 The diagram below illustrates the relation between these classes:
 \begin{center}
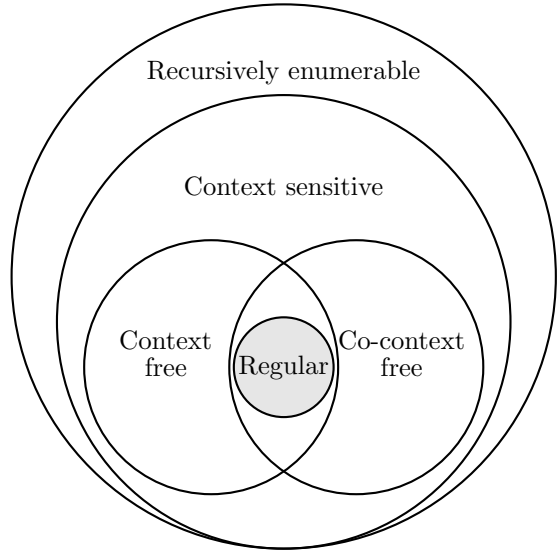

 	\begin{tikzpicture}[scale=1.2]
 		% Outer set: Recursively enumerable
 		\draw[thick] (0,1) circle (3.0);
 		\node at (0,3.25) {Recursively enumerable};
 		
 		% Context-sensitive languages
 		\draw[thick] (0,0.5) circle (2.5);
 		\node at (0,2) {Context sensitive};
 		
 		% Context-free languages
 		\draw[thick] (-0.8,0) circle (1.4);
 		\node at (-1.3,0.3) {Context};
 		\node at (-1.3,0) {free};
 		
 		% Co-context-free languages
 		\draw[thick] (0.8,0) circle (1.4);
 		\node at (1.3,0.3) {Co-context};
 		\node at (1.3,0) {free};
 		
 		% Regular languages (intersection)
 		\draw[thick, fill=gray!20] (0,0) circle (0.55);
 		\node at (0,0) {Regular};
 	\end{tikzpicture}
 	\captionof{figure}{Relations between classes of languages}
 \end{center}
 
 \subsection{Free inverse monoids}
 We now recall the basic definitions and notation concerning inverse monoids, with particular emphasis on free inverse monoids.
 
 Throughout the paper, $X$ will always denote a finite set.
 An \emph{inverse monoid} is a monoid $S$ such that for every element $s \in S$ there exists a unique element $s^{-1} \in S$ satisfying
 \[
 s s^{-1} s = s \quad \text{and} \quad s^{-1} s s^{-1} = s^{-1}.
 \]
 The element $s^{-1}$ is called the \emph{inverse} of $s$. A \emph{free inverse monoid} on a set $X$ is an inverse monoid, which we will denote by $\FIM_X$, together with a map $\iota : X \to \FIM_X$ satisfying the universal property that, for every inverse monoid $S$ and every map $\varphi : X \to S$, there exists a unique homomorphism $\overline{\varphi} : \FIM_X \to S$ such that $\overline{\varphi} \circ \iota = \varphi$.
 
 Let $\tilde{X} = X \cup X^{-1}$, where $X^{-1} = \{x^{-1} \mid x \in X\}$ is a disjoint copy of $X$. We denote by $\tilde{X}^*$ the free monoid on $\tilde{X}$. Given an inverse monoid $S$ generated by $X$, we denote by
 \(
 \pi : \tilde{X}^* \to S
 \)
 the canonical surjective homomorphism. 
 
 We will write $\ell(u)$ to denote the length of a word $u\in \tilde X$ and $|w|$ to represent the length of an element of $\FIM_X$, i.e., the length of a shortest representative of $w$. This notation will be kept throughout the paper.

 As usual, $E(S)$ represents the set of idempotents of $S$, that is,
 \[
 E(S) = \{ e \in S \mid e^2 = e \}.
 \]
 In an inverse monoid, the set $E(S)$ forms a commutative subsemigroup.

Over the years, several distinct notions of conjugacy have been introduced to generalize the standard group-theoretic definition to semigroups and monoids. Significant effort has been devoted to examining, classifying, and comparing these definitions based on their structural properties and suitability for specific classes of semigroups (see, for example, \cite{[KM09],[AKKM17],[Kon18],[AKK19],[ABKKMM26]}). Among these notions, natural conjugacy ($\sim_n$) and $i$-conjugacy ($\sim_i$) have emerged as particularly important extensions. For elements $a, b \in S$, natural conjugacy is defined by $a \sim_n b$ if and only if there exist $g, h \in S^1$ such that
$$ag = gb, \quad bh = ha, \quad hag = b, \quad \text{and} \quad gbh = a.$$
This is an equivalence relation for arbitrary semigroups. On the other hand, in inverse semigroups, we can define $i$-conjugacy by $a \sim_i b$ if and only if there exists $g \in S^1$ such that 
$$g^{-1}ag = b \quad \text{and} \quad gbg^{-1} = a.$$
Crucially, these two definitions are known to coincide in any inverse semigroup \cite[Theorem 2.6]{[Kon18]}.

In this work, we adopt the language and notation of $i$-conjugacy, as it offers a concise and intuitive formulation when operating within inverse settings. Also, we will denote by $[u]$ the conjugacy class of $u$.   We remark that it is proved in \cite{[AKK19]} that the $i$-conjugacy problem is algorithmically decidable in free inverse monoids, leveraging geometric and combinatorial representations such as Munn trees (more details in the next section).  
 
 \subsubsection{Munn trees} \label{sec_munn}
 Given a finite set $X$, the Munn tree of a word in $\tilde X^*$ representing an element on $\FIM_X$ is a birroted tree, i.e., a tree with two fixed vertices. It is constructed by first fixing an initial vertex, and then, for each successive letter of the word, adding a new directed edge labelled by that letter. A directed edge is added unless the letter is the inverse of the label of a directed edge ending at the current vertex; in that case, that edge is crossed and no new edge is added. Once all letters are read, the final vertex is also fixed. For more details see \cite{[M72]}.

Throughout this paper, unless specified otherwise, the term \textit{edge} refers to an undirected edge. Whenever direction is relevant, we shall explicitly refer to \textit{directed edges}.

 For any $u\in X^*$, the Munn tree representing $u\pi$, where $\pi$ is the natural epimorphism between $\tilde X^*$ and $\FIM_X$ will be denoted by $\MT(u)=(\Gamma(u),\alpha(u),\beta(u))$, where $\alpha(u)$ and $\beta(u)$ are the initial and final vertices, respectively. Graphically, we will mark the initial vertex with one incoming arrow and the final one with an outgoing arrow. For example, if $u=aabb^{-1}a$, the triple $(\Gamma(u),\alpha(u),\beta(u))$  is represented by:
 
 
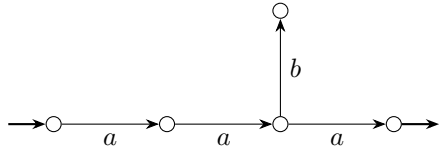
\begin{center}
 	\begin{tikzpicture}[
 		vertex/.style={circle, draw, inner sep=2pt}, % só os nós têm círculos
 		>={Stealth[length=5pt]}
 		]
 		
 		% Nodos (horizontal)
 		\node[vertex] (e) at (0,0) {};
 		\node[vertex] (a1) at (1.5,0) {};
 		\node[vertex] (a2) at (3,0) {};
 		\node[vertex] (a3) at (4.5,0) {};
 		\node[vertex] (b)  at (3,1.5) {}; % ramo vertical
 		
 		% Arestas com rótulos simples
 		\draw[->] (e) -- node[below] {$a$} (a1);
 		\draw[->] (a1) -- node[below] {$a$} (a2);
 		\draw[->] (a2) -- node[below] {$a$} (a3);
 		\draw[->] (a2) -- node[right] {$b$} (b);
 		
 		% Setas de início e\FIM
 		\draw[->, thick] (-0.6,0) -- (e);
 		\draw[->, thick] (a3) -- +(0.6,0);
 		
 	\end{tikzpicture}
 \end{center}

 In \cite{[M72]}, Munn resolved the word problem for the free inverse monoid via these tree representations by proving that two words $u, v \in \tilde{X}^*$ represent the same element in the free inverse monoid generated by $X$ if and only if their Munn trees coincide, i.e., $\mathrm{MT}(u) = \mathrm{MT}(v)$. Consequently, the elements of the free inverse monoid can be identified with, and represented by, these Munn trees.

 Throughout this paper, by a \emph{path} (or \emph{walk}) we mean any sequence of consecutive directed edges in the Munn tree, and by a \emph{simple path} we mean a path that does not repeat vertices (and hence does not repeat undirected edges). 
 
 Since we are working with trees, for any two vertices $p$ and $q$,  there is a unique simple path from $p$ to $q$. 
 The distance between $p$ and $q$ is the length of that (unique) simple path, where by length we mean the number of edges in the  path.

In \cite{[AKK19]}, a characterization of conjugation in free inverse semigroups is established in terms of Munn trees.

\begin{theorem}{\cite[Theorem 7.9]{[AKK19]}}\label{Conj_Munn}
	For any words $u,v$ in $\tilde X^*$, we have that $u\pi$ and $v\pi$ are conjugate if and only if $T(u)=T(v)$ and the label of the simple path in the tree between $\alpha(u)$ and $\alpha(v)$ is equal to the label of the simple path between $\beta(u)$ and $\beta(v)$.      
\end{theorem}

\begin{remark}\label{remark_D_vs_Conj}
In an inverse semigroup $S$, Green's $\mathcal{L}$- and $\mathcal{R}$-relations are defined by $a \mathcal{L} b$ if and only if $a^{-1}a = b^{-1}b$, and $a \mathcal{R} b$ if and only if $aa^{-1} = bb^{-1}$, while  Green's $\mathcal{D}$-relation is given by their composition $\mathcal{D} = \mathcal{L} \circ \mathcal{R} = \mathcal{R} \circ \mathcal{L}$. For more details on the Green's relations see, for example, \cite{[H76]}. By \cite{[M72]}, for any words $u,v \in \tilde{X}^*$, $u\pi$ and $v\pi$ are $\mathcal{D}$-related if and only if $T(u) = T(v)$. Consequently, Theorem~\ref{Conj_Munn} implies that if two elements are conjugate, then they are necessarily $\mathcal{D}$-related.
\end{remark}

Throughout the paper, for the sake of greater fluency in the exposition, we will identify words with their corresponding elements in the free inverse monoid (and so with their Munn trees). In particular, we will refer to a word representing an idempotent simply as an \emph{(geodesic) idempotent word} and for any geodesic words $u$ and $v$, we will say $u$ and $v$ are conjugate instead of $u\pi$ and $v\pi$, where $\pi$ is the natural onto morphism between $X^*$ and $\FIM_X$.  Moreover, for any $g\in\FIM_X$, by the Munn tree representing $g$ we mean the Munn tree representing any word $u\in \tilde X^*$ such that $u\pi=g$.

It is well known that any geodesic word $u$ in a free inverse monoid can be decomposed in the form
\[
u = e_1 u_1 e_2 \cdots e_n u_n e_{n+1},
\]
where $u_1 u_2 \cdots u_n$ is a freely reduced word, $e_i \in E(\FIM_X)\setminus\{1\}$ for $i \in \{2, \ldots, n\}$, and $e_1, e_{n+1} \in E(\FIM_X)$. In fact, this decomposition is unique (see~\cite{[PS05],[CBM26]}).

For future reference, we now formalize several concepts related to Munn trees introduced in {\cite{[CBM26]}}.
\begin{definition}\label{def_main/idempotent}
	Let $u = e_1 u_1 e_2 u_2 \cdots e_n u_n e_{n+1}$ be a geodesic in a free inverse monoid with generating set $X$, and consider its associated Munn tree, where $u_1 u_2 \cdots u_n$ is a freely reduced word, $e_i \in E(\FIM_X)\setminus \{1\}$ for $i \in \{2, \ldots, n\}$, and $e_1, e_{n+1} \in E(\FIM_X)$. Regarding $\MT(u)$, we define:
	\begin{enumerate}
		\item Main path - The unique simple path between the initial and final vertices.
		\item Idempotent edge - Any edge that does not belong to the main path 
		\item Idempotent branch - Every idempotent edge belongs to an idempotent branch and only idempotent edges may belong to an idempotent branch. An idempotent branch may contain several idempotent branches. A new idempotent branch starts in the following situations:
		\begin{enumerate}
			\item First, when an idempotent edge with label $a$ has one of its vertices, say $v$, on the main path; in this case, the branch consists of the entire subtree containing the vertices (and corresponding edges) whose simple path to $v$ crosses that edge labeled $a$.
			Such an idempotent branch (that may contain other idempotent branches) is said to be a \textit{main idempotent branch}.
			\item Second, in a main idempotent branch, every vertex $v$ of degree $g>2$ gives rise to $g-1$ new idempotent branches,
			each starting at one of the edges adjacent to that vertex, except for the edge, say $a$, that belongs to the simple path between $v$ and the main path, i.e., between $v$ and the closest vertex of the main path.
			
			Each such branch starting at an edge of label $b$ with initial vertex $v$ contains all vertices (and corresponding edges) whose simple path to $v$ crosses the edge $b$.
		\end{enumerate}
			\end{enumerate}
			
		\end{definition}
			
		For example, the Munn tree presented below, whose main path is represented by the blue edges, has two main idempotent branches, shown in red and green. Each of these branches contains new idempotent branches. In particular, the vertices $v_1$ and $v_2$ each give rise to two additional branches, while the vertex $v_3$ gives rise to three additional branches.
		
		
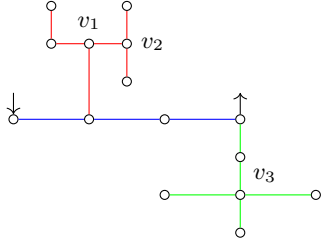
\begin{center}
			\begin{tikzpicture}[
				vertex/.style={circle, draw, inner sep=1.2pt, fill=white},
				every node/.style={font=\small}
				]
				
				% ----------- Vertices -----------
				\node[vertex] (a) at (-2,0) {};
				\node[vertex] (b) at (-1,0) {};
				\node[vertex] (c) at (0,0) {};
				\node[vertex] (d) at (1,0) {};
				\node[vertex,label=above:$v_1$] (f) at (-1,1) {};
				\node[vertex] (g) at (-1.5,1) {};
				\node[vertex] (h) at (-1.5,1.5) {};
				\node[vertex,label=right:$v_2$] (i) at (-0.5,1) {};
				\node[vertex] (j) at (-0.5,1.5) {};
				\node[vertex] (k) at (-0.5,0.5) {};
				\node[vertex] (l) at (1,-0.5) {};
				\node[vertex,label=above right:$v_3$] (m) at (1,-1) {};
				\node[vertex] (n) at (1,-1.5) {};
				\node[vertex] (o) at (0,-1) {};
				\node[vertex] (p) at (2,-1) {};

				%---------- Arestas -----------------
				\draw[blue] (a) -- (b);
				\draw[blue] (b) -- (c);
				\draw[blue] (c) -- (d);
				\draw[red] (b) -- (f);
				\draw[red] (f) -- (g);
				\draw[red] (g) -- (h);
				\draw[red] (f) -- (i);
				\draw[red] (i) -- (j);
				\draw[red] (i) -- (k);
				\draw[green] (d) -- (l);
				\draw[green] (l) -- (m);
				\draw[green] (m) -- (n);
				\draw[green] (m) -- (o);
				\draw[green] (m) -- (p);
				
				\draw[->] ($(a)+(0,0.35)$) -- (a);
				\draw[->] (d) -- ($(d)+(0,0.35)$);
				
			\end{tikzpicture}
		\end{center}

We now present a useful lemma with combines two results in \cite{[CBM26]}.
\begin{lemma}{\cite[Lemma 2.5 + Condition G10]{[CBM26]}}\label{passar3x}
	 A word over $\widetilde X^*$ labels a  geodesic of $\FIM_X$ if and only if it does not cross any edge of its Munn tree more than twice.
 \end{lemma}
 
\begin{remark}\label{rem_vezes}
	It is proved in \cite{[CBM26]} that the path determined by a geodesic in its Munn tree crosses each edge of the main path exactly once, and each edge of the idempotent branches exactly twice.
\end{remark}

\begin{remark}\label{idemp_decomp}
	Note that for a Munn tree representing a geodesic idempotent $e$, all edges belong to a main idempotent branch, and there are as many main idempotent branches as there are edges leaving the initial vertex (which is also the final vertex). Furthermore, each of these main idempotent branches itself represents an idempotent and determines a decomposition $e = f_1 \cdots f_n$, where $n$ is the number of main idempotent branches of the Munn tree and $f_i$ is a geodesic idempotent component associated with the $i$-th main idempotent branch.
\end{remark}

In  \cite{[CBM26]}, the  authors show that the language of all geodesic words in a free inverse monoid is context-free by providing an explicit grammar generating that language. We recall that result as the grammar presented in Section \ref{sec: uconj}.

\begin{theorem}{\cite[Theorem 6.1]{[CBM26]}}\label{grammar_geo}
Let $\FIM_X$ be the free inverse monoid generated by $X$. Then $\Geo(\FIM_X)$ is a context-free language and the following grammar with starting symbol $S$ generates $\Geo(\FIM_X)$.
    \begin{equation*}
	\begin{array}{lrlll}
		& S &\to &  E_{\tilde X\setminus\{x\}}xS_x \mid E_{\tilde X},\ & (x\in \tilde X) \\
		\text{for each } x\in \tilde X, & S_x & \to & E_{\tilde X\setminus\{x^{-1},y\}}yS_y \mid  E_{\tilde X\setminus \{x^{-1}\}} \mid \varepsilon,\ & (y\in \tilde X\setminus\{x^{-1}\}) \\
		\text{for each } \emptyset\neq M\subseteq\tilde X,& E_M & \to & xE_{\tilde X\setminus\{x^{-1}\}}x^{-1} E_{M\setminus\{x\}} \mid \varepsilon,\ & (x\in M).
	\end{array}
\end{equation*}
\end{theorem}

\section{Conjugacy Languages}\label{conj_lang}
 Inspired by the notion of conjugacy languages introduced in \cite{[CHHR16]} for groups, in \cite{[CBM26]}, a definition of a language of representatives associated with an equivalence relation in a semigroup is introduced. Here, we present this definition for free inverse monoids. For any equivalence relation $\sim$ on $\FIM_X$, we define the language
$$\sim\text{--}\,\Geo=\{u\in \tilde X^*: \ell(u)=\min\{\ell(v)\in \tilde X^*: u\pi \sim v\pi\}\}.$$

The following is easy to see for arbitrary semigroups, although we will only present it here for free inverse monoids, as they are the object of study in this paper.

\begin{lemma}\cite[Lemma 3.1]{[CBM26]}
\label{lemma_lang}
	Let $X$ be a finite set. If $\sim_1$ and $\sim_2$ are equivalence relations on $\FIM_X$ such that $\sim_1\,\subseteq\, \sim_2$, then $\sim_2\text{-}\Geo_X(\FIM_X)\subseteq\,\sim_1\text{-}\Geo_X(\FIM_X)$.
\end{lemma}

We now consider the definition in the case where the equivalence relation $\sim$ is the conjugacy $\sim_i$ introduced above. This is the natural semigroup version of the group-theoretic notion of conjugacy languages introduced in \cite{[CHHR16]}.

\begin{definition}
	Let $\FIM_X$ be the free inverse monoid generated by the finite set $X$. We define the following language.
	\begin{equation*}
			\ConjGeo(\FIM_X)= \{w\in \Geo(\FIM_X): \ell(w)\leq \ell(u),\ \forall u\in \tilde X^*\ \text{s.t.}\ [w\pi]=[u\pi]\}.
		\end{equation*}
\end{definition}

Notice that in the monogenic free inverse monoid $\ConjGeo(\FIM_X)=\Geo(\FIM_X)$, and so it is a context-free language by  \cite{[CBM26]}. Indeed, the geodesics representing any two elements in the same conjugacy class have the same length.

\begin{remark}
    By Remark \ref{remark_D_vs_Conj} and Lemma \ref{lemma_lang}, we have $\DGeo(\FIM_X) \subseteq \ConjGeo(\FIM_X)$.
\end{remark}

Consider a Munn tree with initial vertex $p$ and final vertex $q$. By a \emph{shift of the initial and final vertices}, we mean changing the initial vertex from $p$ to $p'$ and the final vertex from $q$ to $q'$ (where $p', q'$ are vertices of the tree) such that the label of the simple path from $p$ to $p'$ is equal to the label of the simple path from $q$ to $q'$.

 \begin{remark}\label{remark_conjgeo}
Given a Munn tree representing an element $g \in \text{FIM}_X$, by Theorem \ref{Conj_Munn}, the conjugates of $g$ in $\text{FIM}_X$ are precisely those elements whose Munn trees can be obtained from that of $g$ via a shift of the initial and final vertices.

Now, by Remark \ref{rem_vezes}, for a fixed Munn tree, the longer the main path, the shorter the length of a geodesic.
Hence, a geodesic is in $\ConjGeo(\FIM_X)$ if there is no shift of the initial and final vertices so that the length of the main path increases.
\end{remark}

In the free group case, both $\Geo(F_n)$ and $\ConjGeo(F_n)$ are regular. In \cite{[CBM26]} it is shown that $\Geo(\FIM_X)$ is both context-free and co-context-free, but not regular. We show that $\ConjGeo(\FIM_X)$ has a greater complexity, by showing that it is neither context-free nor co-context-free (with respect to the standard generating set).

\begin{theorem}
	Let  $|X|\geq 2$ and $\FIM_X$ be the free inverse monoid with basis $X$. Then $\ConjGeo(\FIM_X)$ is not a context-free language.
\end{theorem}

\begin{proof}
	Assume that it is and consider the language
	$$L=\ConjGeo(\FIM_X)\cap (a^{-1})^+(bb^{-1}a^2)^+(abb^{-1}a)^+(a^{-1})^+$$

	Then by assumption $L$ is a context-free language, and so it satisfies the Ogden's Lemma. Let $p$ be the constant from the lemma and $w$ be the word
	$$w=a^{-(2p-1)}(bb^{-1}a^2)^p(abb^{-1}a)^pa^{-1},$$
	whose Munn tree is represented below.
	
	
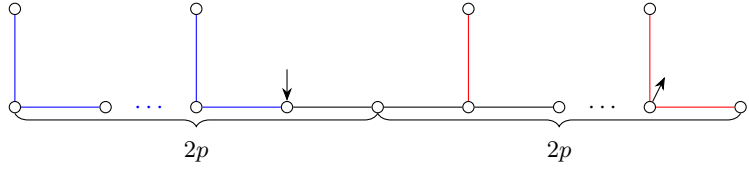
\begin{center}
		\begin{tikzpicture}[
			vertex/.style={circle, draw, inner sep=1.5pt, fill=white},
			>={Stealth[length=5pt]},
			every node/.style={font=\small}
			]
			
			% vértices
			\node[vertex] (v0) at (0,0) {};
			\node[vertex] (v1) at (1.2,0) {};
			\node[vertex] (v2) at (2.4,0) {};
			\node[vertex] (v3) at (3.6,0) {};
			\node[vertex] (v4) at (4.8,0) {};
			\node[vertex] (v5) at (6.0,0) {};
			\node[vertex] (v6) at (7.2,0) {};
			\node[vertex] (v7) at (8.4,0) {};
			\node[vertex] (v8) at (9.6,0) {};
			
			% linha horizontal (em partes)
			\draw[blue] (v0) -- (v1);
			
			\node at (1.8,0) {\color{blue}$\cdots$};
			
			\draw[blue] (v2) -- (v3);
			
			\draw (v3) -- (v4);
			\draw (v4) -- (v5);
			\draw (v5) -- (v6);
			
			\node at (7.8,0) {$\cdots$};
			
			\draw[red] (v7) -- (v8);
			
			% ramos verticais
			\node[vertex] (u0) at (0,1.3) {};
			\node[vertex] (u3) at (2.4,1.3) {};
			\node[vertex] (u6) at (6.0,1.3) {};
			\node[vertex] (u8) at (8.4,1.3) {};
			
			\draw[blue] (v0) -- (u0);
			\draw[blue] (v2) -- (u3);
			\draw[red] (v5) -- (u6);
			\draw[red] (v7) -- (u8);
			
			% setas
			\draw[->] (3.6,0.5) -- (v3);
			\draw[->] (v7) -- (8.6,0.4);
			
			% chavetas
			\draw[decorate,decoration={brace,mirror,amplitude=5pt}]
			(v0.south) -- (v4.south)
			node[midway,below=7pt] {$2p$};
			
			\draw[decorate,decoration={brace,mirror,amplitude=5pt}]
			(v4.south) -- (v8.south)
			node[midway,below=7pt] {$2p$};
		\end{tikzpicture}
	\end{center}

  By Definition~\ref{def_main/idempotent}, in this Munn tree, the main path is the chain of $a$'s between the two rooted vertices (edges in black), with all other edges belonging to main idempotent branches. All main idempotent branches consist of a single edge (edges in red), with the exception of the branch located to the left of the initial vertex (edges in blue). The latter includes multiple edges, and whenever an edge labeled $b$ occurs and we are not at the end of the chain of $a$'s, two new idempotent branches are originated.
    
	First, observe that in order for the elements to stay in $L$, each pumped factor must have one of the following forms: $a^{-k}$, $(bb^{-1}a^2)^k$ or $(abb^{-1}a)^k$, $k\in \N$.  Note that $w$ is clearly a geodesic (by Lemma \ref{passar3x}).

	Also, for any geodesic word of the form $a^{-n_1}(bb^{-1}a^2)^{n_2}(abb^{-1}a)^{n_3}a^{-1}$, for some $n_1,n_2,n_3\in \N$, possible shifts of the initial and final vertices within the chain of $a$'s (the horizontal chain on the representation above)   does not change the length of a geodesic representing the element given by that tree (since the length of the main path remains unchanged),    but if we are able to position both rooted vertices simultaneously on vertices with idempotent branches of $b$'s, it would allow us to shift them up to those branches, increasing the length of the main path, and so decreasing the length of the geodesic word. 
    
    Thus, there are only two different types of shifts: exclusively within the chain of $a$'s (which we will refer to as shifts of the first kind), or into the idempotent branches of $b$'s, preceded by possible shifts in the chain of $a$'s (which we will refer to as shifts of the second kind), with only the second one changing the length of the geodesic word. 
    
    Hence, a word of this form will be in $L$ if a shift of the second kind is not possible.

    We prove that this is the case for $w$ and that for any  decomposition satisfying the conditions of the Ogden's lemma, there will always be a pump leading us to a contradiction.

     Observe that the Munn tree corresponding to $w$ naturally splits into two parts, the first representing the powers of $bb^{-1}a^2$ and the second one representing the powers of  $abb^{-1}a$. Therefore, when indexing the vertices along the bottom chain from left to right, the idempotent branches of $b$'s appear at even positions in the first component and at odd positions in the second.

	Since the parity of the powers of $a^{-1}$ ($2p-1$ and 1) is the same, after any shift of the first type, the initial and terminal vertices lie either both at even vertices or both at odd vertices. 

        Therefore, the word would fail to be in $\ConjGeo$ only if it were possible to place both vertices within the same part of the tree through a shift of the first type.

	However, the distance between the initial and terminal vertices is
	\[
	2p + 2p - (2p-1) - 1 = 2p.
	\]
	Hence this is not possible, and therefore $w \in L$.
	
	Now, we mark the first $2p-1$ $a^{-1}$'s and observe that in the decomposition $w=uvzxy$ of the Ogden's Lemma, where $w_n:=uv^nzx^ny\in L$ for all $n\in \mathbb{N}_0$, $v$ must be a subword of the first power of $a^{-1}$. Since the exponent of $a^{-1}$ is odd, if $\ell(v)$ is even then the first exponent of $a^{-1}$ in $w_n$ will always be odd. If $\ell(v)$ is odd, then the first exponent of $a^{-1}$ in $w_n$ alternates between even and odd. Further, in order to $w_n\in L$ for any $n$, the other pumped word, $x$, must be of the form $(bb^{-1}a^2)^s$ or $(abb^{-1}a)^s$, with $0\leq s\leq p$.
    Otherwise the word would not always be a geodesic, by Lemma \ref{passar3x}. 
    
    If $x$ was a subword of $(abb^{-1}a)^p$, then for a sufficiently large pump the initial vertex would also move into the second part of the Munn tree, and consequently, whenever the exponent is odd, $w_n$ would no longer lie in $L$, since a shift of the second type would be possible.  Therefore, $x$ must be of the form $(bb^{-1}a^2)^s,$ with $ 0\leq s\leq p$.

	Now observe that $\ell(x)$ is always a multiple of $4$, and each pump in the subword $(bb^{-1}a^2)$ adds two edges to the chain of $a$'s in the Munn tree. Moreover, if $2\ell(v) > \ell(x)$, then for sufficiently large $n$ the initial vertex will eventually lie in the second part of the Munn tree,  and passes to the right of the final vertex. At this point, by Lemma~\ref{passar3x}, the word is no longer a geodesic, yielding a contradiction.

    If $2\ell(v) = \ell(x)$, then for any $n$ the Munn tree is analogous to that of $w$, with the part to the left of the initial vertex becoming larger as $n$ increases. Consequently, for any $n\geq 2$, it is possible to place, with a shift of the first type to the left, both the initial and terminal vertices on the same part of the tree. Also, observe that since $\ell(x)$ is a multiple of 4, then $\ell(v)$ has to be even and so the exponent of the first power of $a^{-1}$ is always odd in every $w_n$, $n\geq 0$, which in particular implies that the initial and final vertices stay both at odd positions, and so, for $n\geq 2$, we are able to place them, again through a shift of the same type, simultaneously at vertices with idempotent branches of $b$'s (on the first part of the tree).

    Hence the word does not lie in $L$. Therefore, we must have $2\ell(v) < \ell(x)$.
	
	Further, notice that in this case, for any positive pump, the initial and terminal vertices of $w_n$ lie in different parts of the Munn tree. Therefore, if the first exponent of $a^{-1}$ is even (which occurs when $\ell(v)$ is odd), the word does not lie in $L$. Hence, $r:=\ell(v)$ is even.
    We can then write $\ell(x)=2r+s$ for some $s>0$. Since $\ell(x)$ and $2r$ are even, it follows that $s$ is also even.
	Thus, in $w_0$, the initial vertex will be in the position $2p-1-r$ and the first part of the Munn tree will have length $2p-2r-s$. Hence, from $2p-1-r>2p-2r-s$ we get $r+s>1$, and so
   the initial vertex lies in the second part of the Munn tree and in an odd position (\(2p-1-r\) is odd). Hence it lies in a vertex with an idempotent branch, and therefore the word does not lie in $L$, since both the initial and final vertices have an idempotent branch.
	We thus arrive at a contradiction, allowing us to conclude that $L$, and $\ConjGeo(\FIM_X)$ are not context-free languages.
\end{proof}

\begin{theorem}
	Let  $|X|\geq 2$ and $\FIM_X$ be the free inverse monoid with basis $X$. Then $\ConjGeo(\FIM_X)$ is not a co-context-free language.
\end{theorem}
\begin{proof}
	Assume that it is and consider the language
	$$L=(\tilde X^*\setminus \ConjGeo(\FIM_X))\cap (b^{-1})^+a^{-2}a^4(b^{-1})^+a^{-2}a^2b^+.$$
	
	Then, by assumption, $L$ is a context-free language, and so satisfies the Ogden's Lemma. Let $p$ be the constant from the lemma and $w$ be the word
	$$w=b^{-p}a^{-2}a^4b^{-p}a^{-2}a^2b^p
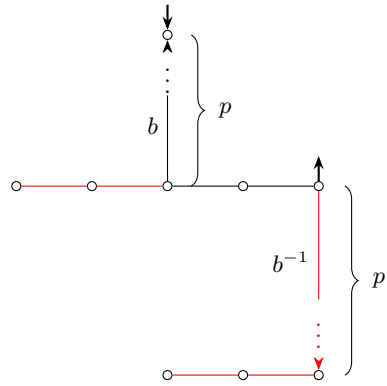
$$
	
	\begin{center}
		\begin{tikzpicture}[
			vertex/.style={circle, draw, inner sep=1.2pt, fill=white},
			>={Stealth[length=5pt]},
			every node/.style={font=\small}
			]
			
			% ---- NÓ CENTRAL ----
			\node[vertex] (v0) at (0,0) {};
			
			% ---- RAMO HORIZONTAL (esquerda) ----
			\node[vertex] (v1) at (-1,0) {};
			\node[vertex] (v11) at (-2,0) {};
			
			\draw[red] (v0) -- (v1);
			\draw[red] (v1) -- (v11);
			
			% ---- RAMO VERTICAL PARA CIMA ----
			\node[vertex] (v2) at (0,2) {};
			\draw[-] (v0) -- ($(v0)!0.6!(v2)$);
			\node at ($(v0)!0.75!(v2)$) {$\vdots$};
			\draw[->] ($(v0)!0.95!(v2)$) -- (v2);
			
			\node[left] at ($(v0)!0.4!(v2)$) {$b$};
			
			\draw[decorate,decoration={brace,mirror,amplitude=5pt}]
			([xshift=7pt]v0.east) -- ([xshift=7pt]v2.east)
			node[midway,right=7pt] {$p$};
			
			% ---- RAMO VERTICAL PARA BAIXO (deslocado para a direita) ----
			\node[vertex] (u1) at (1,0) {};
			\node[vertex] (u2) at (2,0) {};
			\node[vertex] (v3) at (2,-2.5) {};
			
			\draw (v0) -- (u1);
			\draw (u1) -- (u2);
			
			\draw[-, red] (u2) -- ($(u2)!0.6!(v3)$);
			\node at ($(u2)!0.75!(v3)$) {\color{red}$\vdots$};
			\draw[->, red] ($(u2)!0.95!(v3)$) -- (v3);
			
			\node[left] at ($(u2)!0.4!(v3)$) {$b^{-1}$};
			
			\draw[decorate,decoration={brace,mirror,amplitude=5pt}]
			([xshift=8pt]v3.east) -- ([xshift=8pt]u2.east)
			node[midway,right=7pt] {$p$};
			
			% ---- RAMO HORIZONTAL EM BAIXO ----
			\node[vertex] (v4) at (1,-2.5) {};
			\node[vertex] (v5) at (0,-2.5) {};
			
			\draw[red] (v3) -- (v4);
			\draw[red] (v4) -- (v5);
			
			% ---- SETAS DE INÍCIO E\FIM ----
			\draw[->, thick] ($(v2)+(0,0.4)$) -- (v2);
			\draw[->, thick] (u2) -- ++(0,0.4);
			
		\end{tikzpicture}
	\end{center}

In this Munn tree, according to Definition~\ref{def_main/idempotent}, the main path consists of the black edges, while the two main idempotent branches are represented by the red edges. In this case, no new idempotent branches arise within the main idempotent branches.
    
	We claim that $w\in L$. In fact, by Lemma \ref{Conj_Munn}, it is easy to see that the word $w'=a^2b^pb^{-p}a^2b^{-p}a^{-2}$ is conjugated to $w$ and $\ell(w')<\ell(w)$.

	Now, let us mark the first $p$ $b^{-1}$'s. Observe that in order for the elements to stay in $L$, each pumped word must be a power of $b$ or of $b^{-1}$.
	
	Also, we claim that if the exponents of the initial $b^{-1}$ and of $b$ (which represent the distances of the initial and terminal vertices, respectively, to the lower part of their branches of $b$'s, i.e., the distances to the first and the second chains of $a$'s, respectively)  are different, then the word lies in $\ConjGeo(\FIM_X)$. This follows from Theorem \ref{Conj_Munn} and Remark \ref{remark_conjgeo}. 
    If the exponents differ, in most of the cases shifts can only occur along vertical edges, which never increase the length of the main path. The only exception is when the exponent of $b$ exceeds the second exponent of $b^{-1}$ (yielding a tree such as the one illustrated below), where shifts may also proceed along the upper horizontal chain of $a$'s provided the initial and terminal vertices are equidistant from it; this likewise does not increase the length of the main path, and so the word is in $\ConjGeo(\FIM_X)$.

	Therefore, given that under the chosen marking, one of the pumps must occur in the initial power of $b^{-1}$, it follows from the above argument that the second pump must occur in the power of $b$, and both pumps must have the same length to allow the exponents to remain equal.

	Now, considering the decomposition of the words \(w_n = u v^n z x^n y\), where \(v = b^{-\ell}\) and \(x = b^{\ell}\) for some \( \ell > 0 \), obtained from the original decomposition \( w = u v z x y \), we observe that, for a sufficiently large pumping parameter \( n \), we obtain a word whose Munn tree has the following form.
	
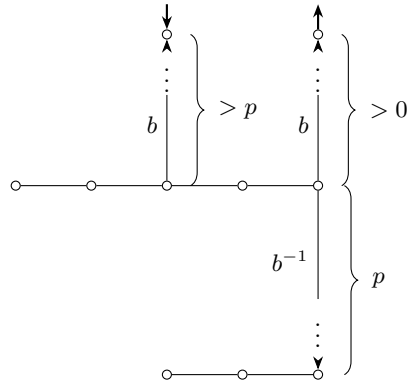
\begin{center}
		\begin{tikzpicture}[
			vertex/.style={circle, draw, inner sep=1.2pt, fill=white},
			>={Stealth[length=5pt]},
			every node/.style={font=\small}
			]
			
			% ---- NÓ CENTRAL ----
			\node[vertex] (v0) at (0,0) {};
			
			% ---- RAMO HORIZONTAL (esquerda) ----
			\node[vertex] (v1) at (-1,0) {};
			\node[vertex] (v11) at (-2,0) {};
			
			\draw (v0) -- (v1);
			\draw (v1) -- (v11);
			
			% ---- RAMO VERTICAL PARA CIMA ----
			\node[vertex] (v2) at (0,2) {};
			\draw[-] (v0) -- ($(v0)!0.6!(v2)$);
			\node at ($(v0)!0.75!(v2)$) {$\vdots$};
			\draw[->] ($(v0)!0.95!(v2)$) -- (v2);
			
			\node[left] at ($(v0)!0.4!(v2)$) {$b$};
			
			\draw[decorate,decoration={brace,mirror,amplitude=5pt}]
			([xshift=7pt]v0.east) -- ([xshift=7pt]v2.east)
			node[midway,right=7pt] {$>p$};
			
			% ---- RAMO VERTICAL PARA BAIXO (deslocado para a direita) ----
			\node[vertex] (u1) at (1,0) {};
			\node[vertex] (u2) at (2,0) {};
			\node[vertex] (v3) at (2,-2.5) {};
			
			\draw (v0) -- (u1);
			\draw (u1) -- (u2);
			
			\draw[-] (u2) -- ($(u2)!0.6!(v3)$);
			\node at ($(u2)!0.75!(v3)$) {$\vdots$};
			\draw[->] ($(u2)!0.95!(v3)$) -- (v3);
			
			\node[left] at ($(u2)!0.4!(v3)$) {$b^{-1}$};
			
			\draw[decorate,decoration={brace,mirror,amplitude=5pt}]
			([xshift=8pt]v3.east) -- ([xshift=8pt]u2.east)
			node[midway,right=7pt] {$p$};
			
			% ---- RAMO HORIZONTAL EM BAIXO ----
			\node[vertex] (v4) at (1,-2.5) {};
			\node[vertex] (v5) at (0,-2.5) {};
			
			\draw (v3) -- (v4);
			\draw (v4) -- (v5);
			
			% ---- RAMO VERTICAL PARA CIMA ----
			\node[vertex] (u3) at (2,2) {};
			\draw[-] (u2) -- ($(u2)!0.6!(u3)$);
			\node at ($(u2)!0.75!(u3)$) {$\vdots$};
			\draw[->] ($(u2)!0.95!(u3)$) -- (u3);
			
			\node[left] at ($(u2)!0.4!(u3)$) {$b$};
			
			\draw[decorate,decoration={brace,mirror,amplitude=5pt}]
			([xshift=7pt]u2.east) -- ([xshift=7pt]u3.east)
			node[midway,right=7pt] {$>0$};
			
			% ---- SETAS DE INÍCIO E\FIM ----
			\draw[->, thick] ($(v2)+(0,0.4)$) -- (v2);
			\draw[->, thick] (u3) -- ++(0,0.4);
		\end{tikzpicture}
	\end{center}
	
	where the power of \( b \) can be made arbitrarily larger than the second power of \( b^{-1} \). Consequently, we observe that as soon as the exponent of $b^{-1}$ exceeds that of $b$ by at least two letters, and given that the initial and final vertices can only move within the branches of the powers of $b$ and advance by two vertices in $a^{-1}$, it immediately follows that $w_n \in \ConjGeo$ in this case. In other words, when this occurs, shifting the initial and final vertices downward reduces the length of the main path, which subsequently increases by at most 4 units. The main path is then maximized in $w_n$ and thus we arrive at a contradiction, allowing us to conclude that $L$, and consequently $ \ConjGeo(\FIM_X)$, are not co-context-free languages.
	
\end{proof}

\section{Elements with a trivial conjugacy class}\label{sec: uconj}

The goal of this section is to show that the ``difficulty" to describe all words in $\ConjGeo$ lies in those representatives corresponding to elements whose conjugacy class is nontrivial. To formalize this idea, we define an equivalence relation as follows: two distinct elements are related if and only if their conjugacy classes are nontrivial. In other words, this relation isolates all elements with trivial conjugacy class, while grouping all remaining elements of the monoid into a single equivalence class. We denote this equivalence relation by UConj\footnote{We call Uconj as it separates elements that are \textbf{u}nique in their conjugacy class while grouping all others by taking their union ($\boldsymbol{\cup}$)}. 

In the case of groups the elements with trivial conjugacy class are the central ones. However,  even though $Z(\FIM_X)=\{1\}$, notice that in any free inverse monoid of rank at least 2 there are ``many'' elements with a trivial conjugacy class. As we will see in Proposition~\ref{ser_conjugado}, the only restriction on the Munn tree of such an element is that there is no directed edge leaving both the initial and the terminal vertex with the same label, that is,  this condition completely describes the Munn trees representing elements with trivial conjugacy classes.

In particular, unlike in the group case, discussed in the last section of this paper, the elements with trivial conjugacy class do not coincide with the elements of the center.

In analogy with the approach taken in \cite{[CBM26]} and in the previous section, we consider the languages of geodesics of minimal length within each equivalence class of UConj.
\begin{definition}
	Let $X$ be a finite set and $\FIM_X$ the free inverse semigroup generated by $X$. We define the language of geodesics $\UConjGeo(\FIM_X)$ as follows.
	\begin{equation*}
			\UConjGeo(\FIM_X)=\{w\in \Geo(\FIM_X): \ell(w)\leq \ell(u), \forall u\in \Geo(\FIM_X) ((w\pi,u\pi)\in \text{UConj})\}.
	\end{equation*}
\end{definition}

Since we are interested in studying language-theoretic properties, the complexity of $\UConjGeo$ is essentially the same as the complexity of the geodesic words representing elements that are only conjugate to themselves. In fact, the difference between the two languages is a set of $2|X|$ words of length $2$.

\begin{proposition}\label{prop_desc_uconj}
	Notice that, for any finite set $X$, we have 
	\begin{equation*}
		\UConjGeo(\FIM_X)=\{w\in \Geo(\FIM_X): [w\pi]=\{w\pi\}\}\cup \{xx^{-1}: x\in \tilde X\}
	\end{equation*}
\end{proposition}
\begin{proof}
	It follows from observing that the conjugacy class of single letter elements are trivial, and so in the non trivial equivalence class of UConj, the shortest representatives must have size two. Within all elements of size 2, the only ones with a non trivial conjugacy classes are the idempotents, and so the equality holds.
\end{proof}

We now present an immediate consequence of Theorem \ref{Conj_Munn}.

\begin{proposition}\label{ser_conjugado}
	For any $w\in \Geo(\FIM_X)$, we have that $w\in \UConjGeo(\FIM_X)$ if and only if either $w$ represents an idempotent of length 2 or the Munn tree representing $w$ has no directed edges starting in the initial and final vertices with the same label.
\end{proposition}

\begin{remark}
    An immediate consequence of this proposition is that a reduced word is in $\UConjGeo(\FIM_X)$ if and only if the first letter is not the inverse of the last one, i.e., if and only if it is cyclically reduced.
\end{remark}

In \cite{[CBM26]}, a description of the Green languages using Munn trees was presented. Specifically, for the monogenic free inverse monoid generated by $\{a\}$, it was shown that $\DGeo$ is $a^* \cup (a^{-1})^*$, that is, the set of elements whose Munn trees have fixed vertices at opposite ends of the chain. Consequently, it was proved that $\DGeo$ is a regular language.
Keeping this observation and Remark \ref{remark_conjgeo} in mind, we present the following immediate proposition.
\begin{proposition}
	For the monogenic free inverse monoid generated by $X=\{a\}$, we have that
	\begin{equation*}
		\begin{aligned}
			\ConjGeo(\FIM)&=\Geo(\FIM)\\
			\UConjGeo(\FIM)&=\DGeo(\FIM)\cup \{aa^{-1}, a^{-1}a\}
		\end{aligned}
	\end{equation*}
	and so in particular $\UConjGeo(\FIM)$ is regular for the monogenic free inverse monoid.
\end{proposition}

In Section \ref{conj_lang}, it is proved that $\ConjGeo(\FIM_X)$ is not context-free for $|X|\geq 2$. We now show that the complexity comes from the classes with two or more elements. Indeed, putting 
$$\ConjGeo_{\geq 2}(\FIM_X):=\{w\in \Geo(\FIM_X): |[w\pi]|\geq 2 \,\wedge\, (\ell(w)\leq \ell(u),\ \forall u\in \tilde X^*\ \text{s.t.}\ [w\pi]=[u\pi])\}$$
we have that $$\ConjGeo(\FIM_X)=\UConjGeo(\FIM_X)\cup \ConjGeo_{\geq 2}(\FIM_X),$$
and we prove that $\UConjGeo(\FIM_X)$ is context-free.

\begin{theorem}
	For any finite set $X$, we have that $\UConjGeo(\FIM_X)$ is a context-free language
\end{theorem}
\begin{proof}
	We prove that $\UConjGeo(\FIM_X)':=\UConjGeo(\FIM_X)\setminus \{xx^{-1}:x\in \tilde X\}$ is context-free. The intended will then follow since context-free languages are closed under finite unions.
	
	Let $$V=\{I\}\cup \{F_M\mid M\subseteq \widetilde X\}\cup \{E_M\mid \emptyset\neq M\subseteq \widetilde X\}\cup  \{S_{y,M}\mid y\in \widetilde X ,\,\emptyset\neq M\subseteq \widetilde X\}$$ be a set of variables. Consider the following grammar with starting symbol $I$, set of terminals $\widetilde X$, set of variables $V$ and the following production rules:

\begin{equation*}
		\begin{aligned}
			I&\to xF_{\tilde X \setminus \{x^{-1}\}} x^{-1} E_{\tilde X\setminus \{x\}}\, |\, xS_{x,\tilde X\setminus\{x\}}, |\, \varepsilon,\ x\in \tilde X,\\
			F_M&\to yF_{\tilde{X}\setminus \{y^{-1}\}}y^{-1}F_{M\setminus \{y\}}\, |\, \varepsilon,\ y\in M,\ \emptyset \neq M\subseteq \tilde{X}, \\
            F_{\emptyset} &\to \varepsilon, \\
			E_M&\to yF_{\tilde{X}\setminus \{y^{-1}\}}y^{-1}E_{M\setminus \{y\}}\, |\, yS_{y,M\setminus\{y\}},\ y\in M,\ \emptyset \neq M\subseteq \tilde{X}, \\
			S_{y,M} &\to F_{\tilde X\setminus \{y^{-1}, x\}}xS_{x,M},\ y\neq x^{-1},\ \emptyset \neq M\subseteq \tilde X, \\
			S_{y,M} &\to F_{M\setminus \{y^{-1}\}},\ y^{-1}\in M,\ \emptyset \neq M\subseteq \tilde X.
		\end{aligned}
	\end{equation*}

	Let $L$ be the context-free language generated by this grammar. We claim that $\UConjGeo(\FIM_X)'=L$.
	
Before diving into the details of the proof, we start with an informal description of the grammar. 

Recall that as mentioned in Remark \ref{idemp_decomp},  any geodesic idempotent can be decomposed in its components, each associated with one main idempotent branch leaving the initial vertex of its Munn tree.

Consider a geodesic word $u$ with standard decomposition $e_1 u_1 e_2 \cdots e_n u_n e_{n+1}$ as described in Section \ref{sec_munn}, i.e., $u_1 u_2 \cdots u_n$ is a freely reduced word, $e_i \in E(\FIM_X)\setminus\{1\}$ for $i \in \{2, \ldots, n\}$, and $e_1, e_{n+1} \in E(\FIM_X)$.

The nonterminals $F_M$ generate exclusively idempotent elements: every derivation starting from $F_M$ produces an idempotent word and no other type of element, with all labels of edges leaving the initial vertices belonging to $M$.  
The nonterminals $E_M$ specify the possible initial letters of each component of the idempotent $e_1$, whenever such components exist, together with the first letter of $u_1$, within the set $M$.

In this way, $E_M$ determines the edges leaving the initial vertex of the Munn tree, while simultaneously identifying the set of all letters that do \emph{not} label edges departing from that vertex (notice that since we are describing elements in $\UConjGeo(\FIM_X)$, Proposition \ref{prop_desc_uconj} must be satisfied).  
Finally, the nonterminals $S_{x,M}$ focus on the construction of the main path, corresponding to the letters of the reduced word, while simultaneously introducing the nonterminals $F_M$ to subsequently generate the idempotent subwords.

We now proceed to the rigorous proof.
    
	First, we will show that $L\subseteq \UConjGeo(\FIM_X)'$.
	To simplify the reading of the proof, we divide it in the proof of several statements that will allow us to reach our goal.
	
    Consider $w\in L$.
	
	\begin{itemize}
		\item Let us justify that $L\subseteq \Geo(\FIM_X)$.
		
		\begin{itemize}
			\item \textit{Statement 1:} For any $\emptyset\neq M\subseteq \tilde X$, we have that any word $e$ in $L(F_M)$ is a geodesic idempotent word with outgoing labels at $\alpha(e)$ in $M$.
			\begin{proof}
				Let $e$ be a word deduced from $F_M$, for some $M\subseteq \tilde{X}$. We prove the intended by induction on the number $k$ of deductions. If $k=1$, then $e=\varepsilon$ and $e$ is a geodesic idempotent word. If $k=3$ we can only obtain words of the form $yy^{-1}$, $y\in \tilde X$, which are also geodesic idempotent words with $y\in M$, by definition of the grammar. Assume that for any $M\subseteq\tilde X$, any word deduced from $F_M$ requiring a number of deductions $\leq k$ is a geodesic idempotent word with outgoing labels at $\alpha(e)$ in $M$. Also, assume that $e$ requires $k+1>1$ deductions.
				Then the first deduction is of the form $F_M\to  yF_{\tilde{X}\setminus \{y^{-1}\}}y^{-1}F_{M\setminus \{y\}}$, for some $y\in M\subseteq \tilde X$. By induction hypothesis, from $F_{\tilde{X}\setminus \{y^{-1}\}}$ and $F_{M\setminus \{y\}}$ we deduce geodesic idempotent words, say $f_1$ and $f_2$, respectively, with the sets $\tilde{X}\setminus \{y^{-1}\}$ and $M\setminus \{y\}$ being in the conditions of the statement.  Now, no component of $f_1$ can start with $y^{-1}$ (since $y^{-1}\notin \tilde{X}\setminus \{y^{-1}\}$) and so $yf_1y^{-1}$ is still a geodesic word representing an idempotent. Similarly, since $y\notin M\setminus\{y\}$, none of the components of $f_2$ start with $y$ and so $yf_1y^{-1}f_2=e$ is a geodesic  word representing an idempotent. Further, the outgoing labels at $\alpha(e)$ are the outgoing labels at $\alpha(f_2)$, which by hypothesis are all in $M\setminus \{y\}$, together with the letter $y$, which is in $M$ by the definition of the grammar, and so, all outgoing labels at $\alpha(e)$ are in $M$, as intended. 
			\end{proof}
			\item \textit{Statement 2:} For any $\emptyset\neq M\subseteq \tilde X$ and any $y\in \tilde X$, we have that any word $u$ in $L(S_{y,M})$ is a geodesic with no outgoing label at $\alpha(u)$ being $y^{-1}$.
			\begin{proof}
				We argue by induction on the number $k$ of deductions. Observe that $k\geq 2$. If $k=2$, then the word is $\varepsilon$, which is a word satisfying the statement. If $k=4$, then it is clear the $u$ is either a single letter that has to be different than $y^{-1}$ or is of the form $xx^{-1}$ where again $x\neq y^{-1}$. Now, assume that the intended is verified for any word requiring up to $k$ deductions and that a word $u$ requires $k+1$ deductions. Then, we have the following possibilities for the first deduction
				\begin{enumerate}
					\item[(A)] $
					S_{y,M}\to F_{\tilde X\setminus \{y^{-1}, x\}}xS_{x,M},\ y\neq x^{-1},\ \emptyset \neq M\subseteq \tilde X
					$
					\item[(B)] $S_{y,M}\to F_{M\setminus \{y^{-1}\}},\ y^{-1}\in M,\ \emptyset \neq M\subseteq \tilde X$
				\end{enumerate}
				
				Notice that case (B) is immediate by \textit{Statement 1}.
				
				Regarding case (A), we can write $u=exv$, where $e$ is deduced from $F_{\tilde X\setminus \{y^{-1}, x\}}$, which is a geodesic word representing an idempotent with no outgoing label at $\alpha(e)$ being $y^{-1}$ ($y^{-1}\notin \tilde X\setminus \{y^{-1}, x\}$), by \textit{Statement 1}, and $v$ is deduced from $S_{x,M}$ requiring up to $k$ deductions, and so is a geodesic word. Let $v=f_1v_1\cdots f_n v_n f_{n+1}$ be its standard decomposition.
				Now, notice that none of the components of $e$ starts with $x$ (since $x\notin \tilde{X}\setminus \{y^{-1},x\}$), which implies that $ex$ is a geodesic word. Similarly, by induction hypothesis, no initial components of $v$, either the first letter of each component of $f_1$ or the first letter of $v_1$ can start with $x^{-1}$, and so $u=exv$ is a geodesic word. Also, observe that the outgoing labels at $\alpha(u)$ are the outgoing labels at $\alpha(e)$ together with $x$. As mentioned before, $y^{-1}$ is not a label of a directed edge leaving $\alpha(e)$ and by definition of the grammar, $x\neq y^{-1}$, and so we get the intended.
			\end{proof}
			\item \textit{Statement 3:} Any word $u$ in $L(E_M)$, for some $\emptyset\neq M\subseteq \tilde X$, is a geodesic with outgoing labels at $\alpha(u)$ in $M$.
			\begin{proof}
				We prove the intended by induction on the number $k$ of deductions. We have that the minimum number of deductions is $k= 3$ and that a word deduced from $E_M$ after only $3$ steps must be a single letter $y$ such that $y\in M$ (and $y^{-1}\in M$, by definition of the grammar):
				$$
				E_M\to yS_{y,M\setminus\{y\}} \to yF_{M\setminus \{y,y^{-1}\}}\to y\varepsilon=y,
				$$
				which is a geodesic satisfying the conditions of the statement.

				Suppose now that any word deduced from $E_M$, for any $M\subseteq \tilde{X}$ using up to $k$ deductions is a geodesic with outgoing labels at the initial vertex in $M$. Further assume that $u$ is a word deduced from $E_M$ requiring $k+1$ deductions, for some $\emptyset\neq M\subseteq \tilde X$. We have two possibilities for the first deduction
				\begin{enumerate}
					\item[(A)] $E_M\to yF_{\tilde{X}\setminus \{y^{-1}\}}y^{-1}E_{M\setminus \{y\}},\ y\in M$
					\item[(B)] $E_M\to yS_{y,M\setminus\{y\}},\ y\in M$
				\end{enumerate}
				In case (A), we can write $u=yey^{-1}v$, where $e$ and $v$ can be deduced from $F_{\tilde{X}\setminus \{y^{-1}\}}$ and $E_{M\setminus \{y\}}$, respectively, requiring up to $k$ deductions. Hence, from \textit{Statement 1}, $e$ is a geodesic representing an idempotent and, by the induction hypothesis, $v$ is a geodesic. 
				
				Now, by \textit{Statement 1}, none of the edges leaving the initial vertex of the Munn tree representing $e$ has label $y^{-1}$ (since $y^{-1}\notin \tilde{X}\setminus \{y^{-1}\}$), and so $yey^{-1}$ is also a geodesic idempotent. Also, by induction hypothesis, all outgoing labels at $\alpha(v)$ are in $M\setminus\{y\}$, and so, if we consider the standard decomposition of $v=f_1v_1\cdots f_nv_nf_{n+1}$ none of the first letters of each component of $f_1$ (if it is non trivial) can be $y$, and that the first letter of $v_1$ cannot be $y$ too, and so $yey^{-1}v$ is also a geodesic. Clearly, by the arguments presented above, all outgoing labels at $\alpha(u)$ are in $M$.

				Regarding case (B), we can write $u=yv$, where $v$ can be deduced from $S_{y,M\setminus\{y\}}$. By \textit{Statement 2}, $v$ is a geodesic, and if we consider the standard decomposition $v=f_1v_1\cdots f_nv_nf_{n+1}$, we have that none of the first letters of each of the components of $f_1$ (when it exists) or the first letter of $v_1$ equals $y^{-1}$, and so $yv$ is also a geodesic. 
			\end{proof}
			
			\item \textit{Statement 4:} All words in $L(I)$ are geodesics.
			\begin{proof}
				We argue by induction on the number of deductions. 
				
				Observe that with one deduction we can only obtain the empty word (which is a geodesic). There are no words deducible with 2 deductions and with 3 we obtain a word of a single letter, which is clearly  geodesic.
				
				Now, assume that any word deduced from $I$ requiring up to $k$ deductions is a geodesic, and let $u$ be a word requiring $k+1$ deductions from $I$. 
				Notice that we have two possibilities for the first deductions, both analogous to the proof of \textit{Statement 3}, and so, the intended follows by induction and reproducing previous arguments.
			\end{proof}
		\end{itemize}
		
		Hence, we have proved that $L\subseteq \Geo(\FIM_X)$.

		\item \textit{Statement 5:} All geodesics in $L(I)$ satisfy Proposition \ref{ser_conjugado}.
			\begin{proof}
				We want to show that in the Munn tree of a geodesic deduced from $I$ there are no directed edges starting in the initial and final vertices with the same label. 

            Clearly, the empty word satisfies Proposition \ref{ser_conjugado}. 
            Let $u=e_1u_1\cdots u_ne_{n+1}$ be a non empty geodesic word deduced from $I$.
            
            Let $R(u):=\{a\in \tilde X:\ a\ \text{is not a label of an outgoing edge at } \alpha(u)\}.$ 
            
            We have two possibilities for the first deduction
            \begin{itemize}
                \item[(A)] $I \to xF_{\tilde X \setminus \{x^{-1}\}} x^{-1} E_{\tilde X\setminus \{x\}},\ x\in \tilde X$

                \item[(B)] $I\to xS_{x,\tilde X\setminus\{x\}},\ x\in \tilde X$
            \end{itemize}

            In case $(A)$, let $e_1=f^1_1f^1_2\cdots f^1_k$, be as described in Remark \ref{idemp_decomp}, let $x_1,\cdots, x_k$ be the first letters of $f^1_1,\cdots, f^1_k$, respectively, and let $y$ be the first letter of $u_1$. We claim that to deduce $u$ we can proceed as follows
            $$I\to^+ e_1yS_{y,R(u)}\to^* u.$$

            The $\to^+$ represents the deductions
            \begin{equation*}
                \begin{aligned}
                    I&\to x_1F_{\tilde X \setminus \{x_1^{-1}\}} x_1^{-1} E_{\tilde X\setminus \{x_1\}}\underbrace{\to^*}_{\text{by Stat. 1}} f^1_1E_{\tilde X\setminus \{x_1\}}\to f^1_1x_2F_{\tilde X\setminus \{x_2^{-1}\}}E_{\tilde X\setminus \{x_1,x_2\}} \\
                    &\underbrace{\to^*}_{\text{By Stat. 1}} f^1_1f^1_2E_{\tilde X\setminus\{x_1,x_2\}}\to^*  e_1E_{\tilde X\setminus\{x_1,\cdots,x_k\} }\to e_1yS_{y,\underbrace{\tilde X\setminus \{x_1,\cdots x_k, y\} }_{=R(u)}}.
                \end{aligned}
            \end{equation*}
            
where the last $\to^*$ is the repetition of the previous steps for all $k$.

Notice that by definition of the grammar, the last deduction represented above is the last one determining an edge leaving $\alpha(u)$.

            Further, observe that the only deduction that will then be influenced by this set $R(v)$ is
            $$S_{y,M} \to F_{M\setminus \{y^{-1}\}},\ y^{-1}\in M,$$
            and by \textit{Statement 1}, from $F_M$ we only deduce idempotent geodesics with outgoing labels at the initial vertex in $M$. Hence, all idempotent branches of $e_{n+1}$,  which start at $\beta(u)$, have label in $R(u)$. Also, the condition $y^{-1}\in M$ in this deduction implies that the inverse of the label of the last edge of the main path, which arrives at the final vertex (which is a label of an edge leaving $R(u)$), has to be in $R(u)$, and so, the outgoing labels at $\beta(u)$ are in $R(u)$.

            Hence, we have
            \begin{equation}
            \begin{aligned}
                 &\text{The labels of the edges leaving $\alpha(u)$ are in $\tilde X\setminus R(u)$,} \\
                &\text{and the labels of the edges leaving $\beta(u)$ are in $R(u)$.}
            \end{aligned}
            \end{equation}

            Therefore, Proposition \ref{ser_conjugado} is satisfied, as intended. 
		\end{proof}
	\end{itemize}
	Therefore, $ L\subseteq \UConjGeo(\FIM_X)'$

	We now prove that $\UConjGeo(\FIM_X)'\subseteq L$. Consider $w\in \UConjGeo(\FIM_X)'$ and its standard decomposition
$$
w=e_1u_1\cdots e_nu_ne_{n+1}.
$$

Analogously to the previous inclusion, we prove the intended by proving a sequence of statements.

\begin{itemize}
	\item \textit{Statement 6:} Any geodesic idempotent is in $L(F_M)$ such that $M\subseteq \tilde{X}$ contains all labels of the edges leaving the initial vertex of the Munn tree representing the idempotent (i.e., the first letter of each component of the idempotent, in the sense of Remark \ref{idemp_decomp}).
		\begin{proof}
		We argue by induction on the length of idempotent words $e$ (recall that $e$ has necessarily even length). If $e$ is the empty word or if $|e|=2$ then it is immediate that $e\in L(F_M)$ with $M$ in the conditions described in the statement. Now, assume that $|e|=2k+2$ and that the result is valid for any $e$ of length at most $2k$.  We can write $e=e_1\ldots e_n$, where each $e_i$ represents one main idempotent branch. Let $M$ be any subset of $\tilde X$ containing the labels of evey directed edge leaving the initial vertex of the Mun tree of $e$ (i.e., containing the first letter of each $e_i$).
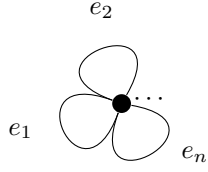
\begin{center}
	\begin{tikzpicture}[scale=1.3]
		
		% Nodo central
		\node[circle, fill=black, inner sep=2.5pt] (v) at (0,0) {};
		
		% Laço esquerdo (e_1)
		\draw (v) edge[loop left, min distance=1.1cm, in=250, out=160]
		node[left=8pt] {$e_1$} (v);
		
		% Laço superior (e_2) -- agora realmente maior
		\draw (v) edge[loop above, min distance=1.1cm, in=160, out=60]
		node[above=8pt] {$e_2$} (v);
		
		% Laço inferior-direito (e_n)
		\draw (v) edge[loop left, min distance=1.1cm, in=350, out=250]
		node[right=8pt] {$e_n$} (v);
		
		% Reticências colocadas entre e_2 e e_n
		\node at (0.3,0.05) {$\cdots$};
		
	\end{tikzpicture}
\end{center}

If $n=1$, then $e$ as only one branch and $e=xfx^{-1}$, with $f\in L(F_{\tilde X\setminus\{x^{-1}\} })$ by induction hypothesis. Observe that since $e$ is a geodesic, no directed edge leaving the initial vertex of the Munn tree of $f$ can have label $x^{-1}$, and so $\tilde X\setminus\{x^{-1}\}$ satisfy the conditions of the statement. Hence, there is a sequence of deductions
$$
F_{M} \to xF_{\tilde X\setminus\{x^{-1}\}} x^{-1} \underbrace{F_{M\setminus \{x\}}}_{\to \varepsilon} \to^+ x f x^{-1}=e,
$$

and so, $e\in L(F_M)$.

Now, if $n>1$, we can write $e_i=x_if_ix_i^{-1}$, where $x_i\in \tilde X$ is the first letter of $e_i$, for all $i$. Further, by the previous reasoning, we have $f_i\in L(F_{\tilde X \setminus \{x_i^{-1}\}})$. Since $e$ is a geodesic, each branch start with different letters, i.e., the letters $x_1,\ldots, x_n$ are all different, and so, by induction hypothesis, any $e_i$ can be deduced from $E_{\tilde X\setminus \{x_j: \text{ for all } j\neq i\}}$. Hence, if $M$ is a subset of $\tilde X$ such that $x_i\in M$, for all $i$, we can consider the deductions
\begin{equation*}
	\begin{aligned}
		F_M&\underbrace{\to}_{x_1\in M} x_1 F_{\tilde X\setminus\{x_1^{-1}\}} x_1^{-1} F_{M \setminus \{x_1\}} \\
		&\underbrace{\to}_{x_2\in M}  x_1 F_{\tilde X\setminus\{x_1^{-1}\}}x_1^{-1} x_2 F_{\tilde X\setminus\{x_2^{-1}\}} x_2^{-1} F_{\tilde X\setminus \{x_1,x_2\}}\to^+ e_1\ldots e_n=e
	\end{aligned}
\end{equation*}
     	\end{proof}
	
	\item \textit{Statement 7:} Any reduced word in  $\UConjGeo(\FIM_X)$ is in $L$.
    
		\begin{proof}
        Recall that a reduced word is in $\UConjGeo(\FIM_X)$ if and only if it is cyclically reduced (i.e., the first letter is not the inverse of the last one).
        
		Let $w=a_1\cdots a_n$ be a reduced word such that $a_1^{-1}\neq a_n$. Observe that since $w$ is reduced we have $a_i^{-1}\neq a_{i+1}$, for all $i$. We now exhibit a sequence of productions that leads to $w$.
		\begin{equation*}
			\begin{aligned}
				I&\to a_1S_{a_1,\tilde X\setminus\{a_1\}} \underbrace{\to}_{a_2\neq a_1^{-1}} a_1F_{\tilde X \setminus \{a_1^{-1},a_2\}}a_2 S_{a_2,\tilde X\setminus\{a_1\}} \to  a_1\varepsilon a_2 S_{a_2,\tilde X\setminus\{a_1\}}\\
				&\to^+ a_1\cdots a_{n}S_{a_{n},\tilde X\setminus\{a_1\}}
				  \underbrace{\to}_{a_n^{-1}\in \tilde X\setminus\{a_1\}} a_1\cdots a_{n}F_{\tilde X\setminus\{a_1, a_n^{-1}\}}\\
				  &\to a_1\cdots a_n \varepsilon = w.
			\end{aligned}
		\end{equation*} 
	\end{proof}
	
		\item \textit{Statement 8:} Any geodesic in $\UConjGeo(\FIM_X)$ of the form $ey$ where $e$ is a geodesic idempotent and $y\in \tilde X$ is in $L(I)$.
        
	\begin{proof}
		We can write $e=f_1\cdots f_k$, where each $f_i$ is a geodesic   representing a main idempotent branch. Let $x_1,\cdots, x_k$ be the first letters of  $f_1,\cdots, f_k$, respectively. Notice that $x_1,\cdots,x_k$ have to be all different.
		
		We start with the following deduction:
		$$
		I\to x_1F_{\tilde X\setminus \{x_1^{-1}\}}x_1^{-1}E_{\tilde X\setminus \{x_1\}}.
		$$
		Now, we can write $f_1=x_1f_1'x_1^{-1}$. Since $f_1$ is a geodesic, no edges leaving the initial vertex of the Munn tree representing $f_1'$ can have label $x_1^{-1}$ and so $f_1'$ can be deduced from $F_{\tilde X\setminus \{x^{-1}\}}$ by \textit{Statement 6}. 
       
        Then we have
		$$
		I\to^+ f_1E_{\tilde X\setminus\{x_1\}}\to f_1x_2F_{\tilde X\setminus \{x_2^{-1}\}}x_2^{-1}E_{\tilde X\setminus \{x_1,x_2\}} \to^+ eE_{\tilde X\setminus \{x_1,\cdots, x_k\}},
		$$
        where the second $\to^+$ represents the deductions made by applying the same reasoning to all other $f_i$.

	   Then, since $ey$ is a geodesic, we have $y\in M:=\tilde X\setminus \{x_1,\cdots, x_k\}$ and so
		$$
		eE_M\to eyS_{y,M\setminus \{y\}}.
		$$
		
		Now, since $ey\in \UConjGeo(\FIM_X)$ and $y^{-1}$ is the label of an outgoing edge at $\beta(ey)$, by Proposition \ref{prop_desc_uconj}, we have $y^{-1}\in M$, and so

		$$
		eyS_{y,M\setminus \{y\}}\to eyF_{M\setminus \{y,y^{-1}\}}\to ey.
		$$
	\end{proof}

	\item \textit{Statement 9:} Let $y$ and $y'$ be the first letter of $u_1$ and the last letter of $u_n$, respectively. We can deduce $w$ as follows
    \begin{equation*}
        \begin{aligned}
            I&\to^+ e_1yS_{y,R(u)} \to^+ e_1u_1\cdots e_nu_nS_{y',R(u)}\to e_1u_1\cdots e_nu_nF_{R(u)\setminus \{y'^{-1}\} } \to^+ w.
        \end{aligned}
    \end{equation*}
		\begin{proof}
        Notice that the first $\to^+$ follows immediately from the proof of \textit{Statement 5}.
        
        For the second $\to^+$ observe the following. The only restriction for a geodesic to belong to $\UConjGeo(\mathrm{FIM}_X)$ is that it satisfies Proposition~\ref{prop_desc_uconj}, which imposes constraints solely on edges having at least one vertex among the rooted vertices. That is, in terms of the standard word decomposition, the restrictions are imposed exclusively on $e_1$, $e_{n+1}$, the first letter of $u_1$, and the last letter of $u_n$. In the remainder of the word (i.e., the middle portion), there are no restrictions whatsoever; hence, any subword may be produced, provided it is guaranteed to be a geodesic. Consequently, we can apply Theorem~\ref{grammar_geo} by observing the following symbol correspondence: what is denoted by $F_M$ in our grammar corresponds to $E_M$ in Theorem~\ref{grammar_geo}, whereas what is represented by $S_{x,M}$ in our grammar is denoted simply by $S_x$ in the context of geodesics.

        The second-to-last deduction is a consequence of the fact that $y'^{-1}\in R(w)$, since $w$ satisfies Proposition \ref{prop_desc_uconj} and $y'^{-1}$ is a label of an edge leaving $\beta(w)$.

        Finally, the last $\to^+$ follows from \textit{Statement 1}. Recall that $w$ satisfies Proposition \ref{prop_desc_uconj}, and so all first letters of each component of $e_{n+1}$ is in $R(w)$.

        We then conclude the intended. 
        \end{proof}
		
\end{itemize}

Hence, we have $\UConjGeo(\FIM_X)'\subseteq L$.

Therefore, $\UConjGeo(\FIM_X)'$ and $\UConjGeo(\FIM_X)$ are context-free languages.
\end{proof}

\begin{remark}
It is a straightforward exercise to prove that $\UConjGeo(\FIM_X)$ is not a regular language. Let $p$ be the pumping length given by the Pumping Lemma. Consider the word $w = a^p a^{-(p! + p)} b b^{-1} a b b^{-1}$. This word belongs to $\UConjGeo(\FIM_X)$. However, since the pumped segment must occur within the first $p$ symbols, it must consist entirely of a non-empty power of $a$. Pumping this segment will alter the exponent of the leading $a$ term. Eventually, this exponent will match the absolute value of the exponent of $a^{-1}$, causing the word to drop out of the language, which yields a contradiction.
\end{remark}

\section{$\UConjGeo$ in groups}
The purpose of this section is to explore   $\UConj$ languages in groups. As noted in the previous section, this corresponds to the language of geodesic words representing central elements, together with a finite set of words. We will study this language in virtually abelian and right-angled Artin groups, proving it is piecewise excluding and piecewise testable, respectively.

\begin{proposition}\label{diffinite}
	Let $G$ be a group. Then
	$$
	\UConjGeo(G)=\{w\in \Geo(G): w\pi\in Z(G)\} \cup F,
	$$
    where $F$ is a finite set of minimal representatives of the elements with a nontrivial conjugacy class.
\end{proposition}

 We start by noting that in a hyperbolic group, this is always a regular languages. This is not surprising, as the center of an non-elementary hyperbolic group is finite. We start with a preliminary lemma.
 
\begin{lemma} \label{fi center}
Suppose that \(\Geo_X(G)\) is regular and that \(Z(G)\) has finite index in \(G\). Then \(\UConjGeo_X(G)\) is regular. 
\end{lemma} 
\begin{proof} 
Since $Z(G)$ has finite index in $G$, the language \[ \pi^{-1}(Z(G)) = \{w\in\widetilde X^* : w\pi\in Z(G)\} \] is regular (see \cite{[Ben79]}).
Consequently, \[ \Geo_X(Z(G)) = \Geo_X(G)\cap\pi^{-1}(Z(G)) \] is regular. 
\end{proof}

\begin{proposition} Let \(G\) be a hyperbolic group and let \(X\) be any finite generating set. Then \(\UConjGeo_X(G)\) is regular. \end{proposition}

\begin{proof} 
The language \(\Geo_X(G)\) is regular for every finite generating set of a hyperbolic group. If \(Z(G)\) is finite, then $\UConjGeo(G)$ is finite, thus regular. Suppose that \(Z(G)\) is infinite. A hyperbolic group with infinite center is virtually cyclic, and in this case \(Z(G)\) has finite index in \(G\). The result therefore follows from the Lemma \ref{fi center}. \end{proof} 

\subsection{Right-angled Artin groups}

Let \(A\) be a finite alphabet. Given a word
\(
u=a_1\cdots a_k\in A^*,
\)
we say that \(u\) is a \emph{piecewise subword} of a word \(w\in A^*\)
if
\(
w\in A^*a_1A^*\cdots A^*a_kA^*.
\)

A language over \(A\) is \emph{piecewise testable} if it is a finite
Boolean combination of languages of the form
\(
A^*a_1A^*\cdots A^*a_kA^*.
\)

A language \(L\subseteq A^*\) is said to be \emph{piecewise excluding}
if there exists a finite set \(W\subseteq A^*\) such that a word belongs
to \(L\) if and only if it contains no element of \(W\) as a piecewise
subword. Every piecewise excluding language is piecewise testable.

We will now prove that, if $G$ is a right-angled Artin groups (RAAG), then $\UConjGeo(G)$ is piecewise testable, thus regular.
For basic facts on RAAGs, we refer the reader to \cite{[Cha07]}.
We remark that that $\Geo(G)$ and $\ConjGeo(G)$ are regular in case $G$ is a RAAG, but not necessarily piecewise testable. Moreover, it is shown in \cite[Theorem 2.5]{[CHHR16]} that \(\ConjGeo(G)\) is the language of geodesic words for which every cyclic permutation is also geodesic.
\begin{theorem}
Let \(A_\Gamma\) be the right-angled Artin group defined by a finite
simplicial graph \(\Gamma\), equipped with its standard generating set
\(X=V(\Gamma)\). Then
\(
\UConjGeo_X(A_\Gamma)
\)
is a piecewise testable language.
\end{theorem}

\begin{proof}
Write
\(
\widetilde X=X\cup X^{-1}.
\)
Let
\[
C=\left\{
x\in X :
x \text{ is adjacent to every vertex of }X\setminus\{x\}
\right\}.
\]
It is well known that the center of \(A_\Gamma\) is the special subgroup
generated by \(C\) (see \cite[Section 2.3]{[Cha07]}); that is,
\(
Z(A_\Gamma)=A_C
\)
and \(A_C\) is a free abelian
group with standard generating set \(C\).

It is enough to prove that the language
\(\Geo_X\bigl(Z(A_\Gamma)\bigr)
\)
of geodesic words over \(\widetilde X\) representing elements of
\(Z(A_\Gamma)\) is regular.

Consider the canonical homomorphism
\(
\rho_C\colon A_\Gamma\longrightarrow A_C,
\)
induced by
\[
\rho_C(x)=
\begin{cases}
x, & x\in C,\\
1, & x\in X\setminus C.
\end{cases}
\]
Let
\(
\widehat{\rho}_C\colon \widetilde X^*\longrightarrow \widetilde C^*
\)
be the induced monoid homomorphism, where
\(\widetilde C=C\cup C^{-1}\), given by
\[
\widehat{\rho}_C(x^{\pm1})=
\begin{cases}
x^{\pm1}, & x\in C,\\
\varepsilon, & x\in X\setminus C.
\end{cases}
\]

Suppose that \(w\in\widetilde X^*\) is geodesic and represents an
element \(g\in A_C\). Since \(\rho_C\) restricts to the identity on
\(A_C\), the word \(\widehat{\rho}_C(w)\) also represents \(g\). If
\(w\) contained a letter from \(\widetilde X\setminus\widetilde C\),
then
\(
\bigl|\widehat{\rho}_C(w)\bigr|<|w|,
\)
contradicting the geodesicity of \(w\). Hence every geodesic word
representing an element of \(A_C\) lies in \(\widetilde C^*\).

Conversely, suppose that \(w\in\widetilde C^*\) is geodesic in \(A_C\).
If \(w\) were not geodesic in \(A_\Gamma\), there would exist a word
\(v\in\widetilde X^*\) representing the same element as \(w\) such that
\(
|v|<|w|.
\)
Applying the retraction, the word \(\widehat{\rho}_C(v)\) would represent
the same element of \(A_C\) as \(w\), and
\[
\bigl|\widehat{\rho}_C(v)\bigr|
\leq |v|
<|w|,
\]
contradicting the geodesicity of \(w\) in \(A_C\). Therefore
\[
\Geo_X\bigl(Z(A_\Gamma)\bigr)
=
\Geo_C(A_C).
\]
Since \(A_C\) is finitely generated abelian,
\cite[Proposition~6.2]{[HHR07]} implies that
\(\Geo_C(A_C)\) is piecewise excluding, and hence
piecewise testable. By Proposition \ref{diffinite},
\[
\UConjGeo_X(A_\Gamma)
=
\Geo_X(Z(A_\Gamma))\cup F,
\]
where \(F\) is finite. Since finite languages are piecewise testable and
the class of piecewise-testable languages is closed under finite unions,
\(\UConjGeo_X(A_\Gamma)\) is piecewise testable.
\end{proof}

\subsection{Virtually abelian groups}

In \cite{[HHR07]}, it is shown that, for a virtually abelian group, there is a generating set for which the language of geodesics is piecewise excluding. In \cite{[CHHR16]}, it is shown that, in this case, there is a generating set for which $\ConjGeo(G)$ is piecewise testable. We follow that work and prove that there is a generating set for which $\UConjGeo(G)$ is piecewise excluding.

\begin{lemma}\cite[Lemma 4.3.2]{[ECHLPT92]}\label{dickson}
	Equip \(\mathbb N^r\) with the componentwise partial order
	\[
	(m_1,\ldots,m_r)\preceq(n_1,\ldots,n_r)
	\quad\Longleftrightarrow\quad
	m_i\leq n_i
	\text{ for every }i.
	\]
	Every subset of \(\mathbb N^r\) has only finitely many minimal
	elements.
\end{lemma}

For a finite inverse-closed generating set \(S\) of a group \(G\) and
a subgroup \(K\leq G\), we write
\[
\Geo_S(K)
=
\{w\in\Geo_S(G):w\pi\in K\}.
\]

\begin{theorem}
	Let \(G\) be a finitely generated virtually abelian group. Then there exists
	a finite inverse-closed generating set \(S\) of \(G\) such that
	\(
	\Geo_S(Z(G))
	\)
	is piecewise excluding. Consequently,
	\(\UConjGeo_S(G)\) is piecewise excluding, and hence regular.
\end{theorem}

\begin{proof}  
	Put
	\(
	H=Z(G).
	\)
	Since \(G\) is virtually abelian, there exists a finite-index normal abelian
	subgroup \(A\trianglelefteq G\). Set
	\(
	N=AH.
	\)
	Since \(H\) is central, \(N\) is again a finite-index normal abelian subgroup
	of \(G\), and \(H\leq N\).
	
	We use the generating set constructed in the proof of
	\cite[Proposition 6.3]{[HHR07]}. Namely, starting with a finite
	generating set of \(N\) containing a generating set for \(H\), one obtains
	a finite inverse-closed generating set
	\(
	S_0=X_0\cup Y
	\)
	of \(G\), where

	\begin{enumerate}
		\item \(X_0\subseteq N\) and \(Y\subseteq G\setminus N\);
		\item both \(X_0\) and \(Y\) are inverse-closed;
		\item \(X_0\) is closed under conjugation by elements of \(Y\);
		\item \(Y\) contains a representative of every nontrivial coset
		of \(N\) in \(G\);
		\item whenever
		\(
		w=_G xy,
		\)
		where \(w\) is a word of length at most \(3\) over \(Y\),
		\(y\in Y\cup\{1\}\), and \(x\in N\), then \(x\in X_0\).
	\end{enumerate} 
	 The
	argument given there shows that every geodesic representing an element of
	\(N\) contains no letters from \(Y\).
	
	We now modify \(X_0\) in order to obtain the analogous property for \(H\).
	Write
	\[
	X_0\setminus H=\{b_1,\ldots,b_m\}.
	\]
	If \(m=0\), no modification is needed. Otherwise, consider
	\[
	\mathcal R=
	\left\{
	(n_1,\ldots,n_m)\in\mathbb N^m
	\;\middle|\;
	b_1^{n_1}\cdots b_m^{n_m}\in H
	\right\}.
	\]
	By Lemma~\ref{dickson},
	\(\mathcal R\setminus\{\mathbf 0\}\) has finitely many minimal elements;
	denote their set by \(\mathcal M\). For each
	\(\mathbf n=(n_1,\ldots,n_m)\in\mathcal M\), put
	\(
	h_{\mathbf n}=b_1^{n_1}\cdots b_m^{n_m}\in H,
	\)
	and define
	\[
	X=
	X_0\cup
	\left\{
	h_{\mathbf n}^{\pm1}
	\;\middle|\;
	\mathbf n\in\mathcal M,\ h_{\mathbf n}\neq1
	\right\},
	\qquad
	S=X\cup Y.
	\]
	Since the new generators are central, the properties   are preserved.
	In particular,
	\begin{equation}\label{eqvirtab1}
	w\in\Geo_S(G),\quad w\pi\in N
	\quad\Longrightarrow\quad
	w\in X^*.
	\end{equation}
	
	Let
	\(
	C=S\cap H=X\cap H.
	\)
	We claim that
	\begin{equation}\label{eqvirtab2}
	\Geo_S(H)=\Geo_C(H).
\end{equation}
	Let \(w\in\Geo_S(H)\). By (\ref{eqvirtab1}), we have \(w\in X^*\).
	Suppose that \(w\) contains a letter from \(X\setminus H\). Since the
	generators added to \(X_0\) belong to \(H\),
	\[
	X\setminus H=X_0\setminus H=\{b_1,\ldots,b_m\}.
	\]
	Let \(r_i\) be the number of occurrences of \(b_i\) in \(w\), and put
	\(
	\mathbf r=(r_1,\ldots,r_m).
	\)
	
	As \(N\) is abelian and \(w\pi\in H\), we have
	\[
	b_1^{r_1}\cdots b_m^{r_m}\in H.
	\]
	Thus
	\(
	\mathbf r\in\mathcal R\setminus\{\mathbf0\}.
	\)
		Choose \(\mathbf n=(n_1,\ldots,n_m)\in\mathcal M\) with
	\(\mathbf n\leq\mathbf r\) componentwise. Notice that
	\(
	n_1+\cdots+n_m\geq2,
	\)
	since otherwise \(b_i\in H\) for some \(i\).
	
	Since \(N\) is abelian,  we can replace \(n_i\) occurrences of all the \(b_i\) in
	\(w\)  by
	\(
	h_{\mathbf n}=b_1^{n_1}\cdots b_m^{n_m}\in S.
	\)
	
	Since $	n_1+\cdots+n_m\geq2$, at least two letters are replaced by one.
	In either case, we obtain a shorter representative of \(w\pi\), a
	contradiction. Hence \(w\in C^*\).
	
	Conversely, suppose that \(w\in\Geo_C(H)\) is not geodesic
	with respect to \(S\). Let \(v\) be a shorter \(S\)-geodesic representing \(w\pi\).
	By the previous argument \(v\in C^*\), contradicting
	the geodesicity of \(w\) in \(H\). This proves (\ref{eqvirtab2}).
	
	Since \(H\) is finitely generated abelian, it follows from
	\cite[Proposition~6.2]{[HHR07]} that
	\(
	\Geo_C(H)
	\)
	is piecewise excluding. Hence, by (\ref{eqvirtab2}),
	\(\Geo_S(Z(G))\) is piecewise excluding.

It remains to consider \(\UConjGeo_S(G)\). By
Proposition~5.1,
\[
\UConjGeo_S(G)
=
\Geo_S(H)\cup F,
\]
where \(F\) is the set of shortest representatives of the noncentral
\(\UConj\)-class. If \(G\) is abelian, then \(F=\varnothing\)
and the result follows immediately. Otherwise, the shortest noncentral
elements have length one, and therefore
\(
F=S\setminus C.
\)

Let \(W\subseteq C^*\) be a finite set such that
\(\Geo_C(H)\) consists precisely of the words over \(C\)
which contain no element of \(W\) as a piecewise subword. Put
\[
W'
=
W\cup
\{xy\in S^2:x\notin C\text{ or }y\notin C\}.
\]
We claim that \(\UConjGeo_S(G)\) is precisely the language
of words over \(S\) which avoid the elements of \(W'\) as piecewise
subwords.

Indeed, a word over \(C\) avoids \(W'\) if and only if it avoids \(W\),
and hence belongs to \(\Geo_C(H)\). On the other hand, a
word containing a letter from \(S\setminus C\) avoids \(W'\) if and only
if it consists of that single letter. Thus the language defined by
excluding \(W'\) is
\[
\Geo_C(H)\cup(S\setminus C)
=
\UConjGeo_S(G).
\]
Therefore \(\UConjGeo_S(G)\) is piecewise excluding, and
hence regular.
\end{proof}

From our results in the free inverse monoid and in hyperbolic and right-angled Artin groups, one might be led into thinking that the complexity of $\UConjGeo$ is bounded above by the complexity of $\ConjGeo$. We show an example (from \cite{[CHHR16]}) where, for a given generating set, $\UConjGeo$ is not regular, but $\ConjGeo$ is.

\begin{example}
	Consider the following example from \cite[Proposition 5.1]{[CHHR16]}. Let
	\[
	G=
	\left\langle
	a,b,t
	\;\middle|\;
	[a,b]=1,\quad t^2=1,\quad tat=b
	\right\rangle
	\cong
	\mathbb Z^2\rtimes\mathbb Z/2\mathbb Z,
	\]
 with generating set
	\(
	Z=\{a^{\pm1},b^{\pm1},t\}.
	\)
	By \cite[Proposition~5.1]{[CHHR16]},
	\(\ConjGeo_Z(G)\) is regular.
	
	Every element of \(G\) can be written as
	\[
	a^i b^j t^\varepsilon,
	\qquad
	i,j\in\mathbb Z,\quad \varepsilon\in\{0,1\}.
	\]
	Since
	\(
	ta^ib^jt=a^jb^i,
	\)
	an element \(a^ib^j\) is central if and only if \(i=j\), while an element
	\(a^ib^jt\) does not commute with \(a\). Hence
	\(
	Z(G)=\langle ab\rangle.
	\)
	
	Now consider the regular language
	\(
	R=a^+b^+.
	\)
	Every word \(a^pb^q\), with \(p,q\geq1\), is geodesic. Indeed, we can define a homomorphism
	\[
	\chi\colon G\longrightarrow\mathbb Z,
	\qquad
	\chi(a)=\chi(b)=1,\quad \chi(t)=0.
	\]
	Then, for every word \(w\) over \(Z\), we have that 
	\(
	|w|\geq|\chi(w\pi)|
	\)
	 and 
	\(\chi(a^pb^q)=p+q\).
	
By Proposition~5.1, the words of \(\UConjGeo_Z(G)\) of length greater than one represent central elements. Therefore \[ \UConjGeo_Z(G)\cap R = \{a^nb^n:n\geq1\}. \] This language is not regular. Since \(R\) is regular, it follows that \(\UConjGeo_Z(G)\) is not regular. Thus, even for virtually abelian groups, regularity of \(\ConjGeo\) does not imply regularity of \(\UConjGeo\) with respect to the same generating set.
\end{example}
`

\section*{Acknowledgments}

The first author was supported by national funds through the Fundação para a Ciência e Tecnologia, FCT, under the project
UID/04674/2025. The second author is supported by national funds through the FCT – Fundação para a Ciência e a Tecnologia, I.P., under the scope of the individual research grant 2025.03264.BD and the projects UID/297/2025 and UID/PRR/297/2025 (Center for Mathematics and Applications - NOVA Math).
 
\bibliographystyle{plain}
\bibliography{Bibliografia}
 \end{document}